\documentclass[11pt]{amsart}
\usepackage[utf8]{inputenc}
\usepackage[english]{babel}

\usepackage{amscd,amssymb}
\usepackage{amsmath}
\usepackage{amsthm}
\usepackage{bbm}
\usepackage{graphicx}
\usepackage{fancybox}
\usepackage{epic,eepic}
\usepackage{amstext}
\usepackage{mathtools}
\usepackage{esint}
\usepackage[square,numbers]{natbib}
\usepackage{booktabs}
\usepackage[hide links]{hyperref}
\usepackage{comment}
\usepackage{mathtools}
\usepackage{tikz,pgfplots}
\usepackage{subcaption}
\usepackage{enumitem}
\usepackage{todonotes}

\usepackage[noabbrev, capitalise, nameinlink]{cleveref}
\crefname{equation}{}{}
\crefname{assumption}{Assumption}{Assumptions}
\crefformat{equation}{\textup{#2(#1)#3}}

\usepackage{algorithm}
\usepackage{algpseudocode}
\usepackage{tabularx}
\newcolumntype{Y}{>{\centering\arraybackslash}X}

\newtheorem{theorem}{Theorem}[section]
\newtheorem{corollary}[theorem]{Corollary}
\newtheorem{lemma}[theorem]{Lemma}

\theoremstyle{definition}

\theoremstyle{remark}
\newtheorem{remark}[theorem]{Remark}
\numberwithin{theorem}{section}
\numberwithin{equation}{section}
\numberwithin{table}{section}
\numberwithin{figure}{section}

\newcommand{\R}{\mathbb R}
\newcommand{\N}{\mathbb N}
\renewcommand{\P}{\mathbb P}
\newcommand{\B}{\mathbb B}
\newcommand{\T}{\mathcal T}
\newcommand{\A}{\mathcal A}
\newcommand{\I}{\mathcal I}
\newcommand{\Q}{\mathcal Q}
\newcommand{\U}{\mathbf U}
\newcommand{\V}{\mathbf V}
\newcommand{\W}{\mathcal W}
\newcommand{\Z}{\mathbf Z}
\newcommand{\F}{\mathbf F}

\newcommand{\Tau}{{\boldsymbol\tau}}
\newcommand{\proj}{\mathcal P}

\newcommand{\ufd}{u_{h,\Tau}}

\newcommand{\norm}[2]{\left\lVert#1\right\rVert_{#2}}
\newcommand{\seminorm}[2]{\left |#1\right |_{#2}}
\newcommand{\normH}[1]{\left |#1\right |_{1,h}}
\newcommand{\normL}[1]{\left\lVert#1\right\rVert_{0,h}}
\newcommand{\normE}[1]{\left\lVert#1\right\rVert_{\star}}
\newcommand{\vertiii}[1]{{\left\vert\kern-0.25ex\left\vert\kern-0.25ex\left\vert #1  \right\vert\kern-0.25ex\right\vert\kern-0.25ex\right\vert}}
\renewcommand{\sp}[2]{\left(#1, #2\right)_{L^2(\Omega)}}
\newcommand{\spE}[2]{\left(#1, #2\right)_{L^2(E)}}
\newcommand{\spT}[2]{\left(#1, #2\right)_{L^2(I_n)}}
\newcommand{\jump}[2]{[#1]_{#2}}

\newcommand\ds{\,\text{d}s}
\newcommand\dt{\,\text{d}t}
\newcommand\dx{\,\text{d}x}

\allowdisplaybreaks

\begin{document}

\begin{abstract}
A novel space-time discretization for the (linear) scalar-valued dissipative wave equation is presented. It is a structured approach, namely, the discretization space is obtained tensorizing the Virtual Element (VE) discretization in space with the Discontinuous Galerkin (DG) method in time. As such, it combines the advantages of both the VE and the DG methods. The proposed scheme is implicit and it is proved to be unconditionally stable and accurate in space and time. 
\end{abstract}
	
\title[A DG-VEM method for the dissipative wave equation]{A DG-VEM method for the dissipative wave equation}
\author[P.F.~Antonietti, F.~Bonizzoni, M.~Verani]{Paola Francesca Antonietti$^*$, Francesca Bonizzoni$^*$, Marco Verani$^*$}
\address{${}^*$ MOX - Department of Mathematics, Politecnico di Milano, via Bonardi 9, 20133 Milano, Italy}
\email{paola.antonietti@polimi.it, francesca.bonizzoni@polimi.it, marco.verani@polimi.it}
\thanks{PFA, FB and MV are members of the INdAM Research group GNCS.
FB is partially funded by ``INdAM – GNCS Project", codice CUP\_E53C22001930001.
PFA and MV are partially funded by PRIN2017 n. 201744KLJL and PRIN2020 n. 20204LN5N5, funded by the Italian Ministry of Universities and Research (MUR). PFA is partially funded by Next Generation EU. The present research is part of the activities of “Dipartimento di Eccellenza 2023-2027".}

\maketitle

%

\section{Introduction}

In this paper we propose a space-time Virtual Element/Discontinuous Galerkin method for the (linear) scalar-valued dissipative wave equation in two- and three-dimensions. The method is based on Virtual Element (VE) for space discretization coupled with  discontinuous Galerkin (DG) finite element method for the time integration of the resulting second-order ordinary differential system. The model problem considered in this paper serves as a prototype model for the vector-valued (damped) elastic wave equation typically encountered in geophysical applications.

The Virtual Element method (VEM) has been introduced in \cite{Beirao-Brezzi-Cangiani-Manzini-Marini-Russo2013} for elliptic problems. VEMs for linear and nonlinear elasticity have been developed in \cite{BeiraoBrezziMarini_2013,GainTalischiPaulino_2014,BeiraoLovadinaMora_2015}, whereas VEMs for parabolic, plate bending, Cahn-Hilliard, Stokes,  Helmholtz and Laplace-Beltrami problems have been addressed in \cite{Vacca-Beirao2015,BrezziMarini_2013, AntoniettiBeiraoScacchiVerani_2016,AntoniettiBeiraoMoraVerani_2014,PerugiaPietraRusso_2016,SVEMbasic}.
VEMs for the space discretization of wave-type problems have been addressed in \cite{Vacca2017, Antonietti_et_al_2021, Antonietti_et_al_2022, Dassi_et_al_2022}.

Concerning time-integration of second-order differential systems stemming from space discretization of wave-type problems, classically, time marching schemes are based on (either implicit or explicit) finite differences approaches, e.g., we refer to  \cite{Ve07,Bu08} for a a general overview. On the other hand, space-time finite element methods for second-order hyperbolic equations have been largely developed, thanks to their ability to achieve high-order approximations in both space and time,  to accurately capture steep fronts, and their firm mathematical foundation, where stability and convergence can been proved.

Among space-time finite element methods, we can distinguish between ``structured" and ``unstructured" numerical approaches. In ``structured" approaches, the space-time grid is obtained as tensor product of space and time meshes; a non-exhaustive list of approaches includes \cite{Tezduyar06,StZa17,ErWi19,BaMoPeSc20,Antonietti-Mazzieri-Migliorini2020}. For such formulations, $h-$, $p-$ or $hp-$ adaptive refinement of the space-time domain can be  developed and implemented, see, e.g., \cite{georgoulis2021,CANGIANI2019}.  
On the other hand, ``unstructured" techniques, see, e.g., the seminal works \cite{Hughes88,Idesman07} make use of full space-time meshes, where time is treated as an additional dimension, see \cite{Yin00,AbPeHa06,DoFiWi16} for examples, and the recent contribution~\cite{gomez2022}. 
Among unstructured methods, we also mention Trefftz-type techniques \cite{KrMo16,BaGeLi17,BaCaDiSh18,MoiolaPerugia_2018,PerugiaSchoberlStocker_2020}, in which the numerical solution is looked for in the Trefftz space, and the tent-pitching paradigm \cite{GoScWi17}, in which the space-time elements are progressively built on top of each other in order to grant stability of the numerical scheme. Recently, in \cite{MoiolaPerugia_2018,PerugiaSchoberlStocker_2020} a combination of Trefftz and tent-pitching techniques has been proposed with application to first order hyperbolic problems. A tent-pitching scheme motivated by Friedrichs' theory can be found in \cite{GopalakrishnanMonkSepulveda_2015}.

The DG approach has been extensively used to approximate initial-value problems, where the DG paradigm shows certain advantages with respect to other implicit schemes such as the Johnson's method, see e.g. \cite{JOHNSON1993,ADJERID2011}. Indeed, thanks to the DG paradigm, the solution  at time-slab $[t_n,t_{n+1}]$ depends only on the solution at the time instant $t_n^-$. The use of DG methods in both space and time dimensions leads to a fully DG space-time formulation such as e.g., \cite{Delfour81,Vegt2006,WeGeSc2001,Antonietti-Mazzieri-Migliorini2020}.

Finally, a typical approach for second order differential equations consists in reformulating them as a system of first order hyperbolic equations. Thus, velocity is considered as an additional problem's unknown, yielding to doubling the dimension of the final linear system, cf. \cite{Delfour81,Hughes88,FRENCH1993,JOHNSON1993,ThHe2005,AntoniettiMazzieriMigliorini_2023}.

In this work we present a novel structured VEM/DG formulation that combines the VE advantages for space discretization together with those of the DG methods for time integration. The obtained scheme is implicit, unconditionally stable and provides an accurate approximation with respect to both space and time discretization errors.
Throughout the paper we will use the notation $x\lesssim y$ with the meaning $x\leq c y$, with $c$ positive constant independent of the discretization parameters. 

The paper is organized as follows. In Section~\ref{sec:problem_setting} we introduce the problem; its semi-discrete VEM approximation is discussed in Section~\ref{sec:VEM_space_discretization}, and in Section~\ref{sec:DG_time} we present DG discretization in time. Section~\ref{sec:VEM-DG} introduces the fully-discrete VEM-DG formulation and studies its well-posedness and stability, whereas in Section~\ref{sec:error_analysis} we prove \textit{a priori} error estimates in a suitable energy norm. Finally, in Section~\ref{sec:numerical_tests}, the method is validated through several numerical experiments in two dimensions (in space).

\section{Problem setting}
\label{sec:problem_setting}

Let $\Omega\subset\R^d$, $d=2,3$, be an open bounded convex polygonal domain. The problem we are interested reads as follows: for $T>0$, find $u\colon \Omega\times (0,T]\rightarrow\R$ such that
\begin{equation}
	\label{eq:pde}	
	\left\{\begin{array}{ll}
		u_{tt}+ \nu u_t-\Delta u=f,&\text{in }\Omega\times (0,T],\\
		u=0, &\text{on }\partial\Omega\times (0,T],\\
		u(\cdot,0)=u_0,\, u_t(\cdot,0)=z_0,&\text{in }\Omega,
	\end{array}\right.
\end{equation}
where $\nu\in\R^+$ is the dissipation coefficient, $f$ is the external force, $u_0$ and $z_0$ are the initial data, and $u_t,\, u_{tt}$ denote the first and second order time derivative of the unknown function $u$, respectively. Note that, by little modifications, our analysis extends to the case of (positive) bounded dissipation function $\nu\in L^\infty(\Omega)$.
By standard arguments, we derive the variational formulation of~\eqref{eq:pde}: given $f\in L^2(\Omega\times(0,T))$ and $u_0,\, z_0\in H^1_0(\Omega)$, find $u\in C^0\left(0,T; H^1_0(\Omega)\right)\cap C^1\left(0,T;L^2(\Omega)\right)$ such that, for all $v\in H^1_0(\Omega)$ and for a.e. $t\in (0,T)$
\begin{equation}
	\label{eq:pde_weak}
	\sp{u_{tt}(t)}{v} + \nu \sp{u_{t}(t)}{v} + a(u(t),v) = \sp{f(t)}{v},
\end{equation}
supplemented with the initial conditions $u(\cdot,0)=u_0$, $u_t(\cdot,0)=z_0$, where $\sp{\bullet}{\bullet}$ denotes the $L^2(\Omega)$-inner product, and $a\colon H^1_0(\Omega)\times H^1_0(\Omega)\rightarrow\R$ is defined as $a(w,v)=\sp{\nabla w}{\nabla v}$. 
Following \cite{Duvant-Lions2012} it is possible to prove existance and uniqueness of the solution to problem~\eqref{eq:pde_weak}.

\section{Space discretization based on the VEM}
\label{sec:VEM_space_discretization}

In this section we apply the VEM to discretize problem~\eqref{eq:pde_weak} in space. In particular, we follow~\cite{Vacca2017}, where the VE space discretization of~\eqref{eq:pde_weak} with damping $\nu=0$ is introduced. We start recalling the ingredients of the VEM that we will need, with focus on the two-dimensional case (the three-dimensional case being analogous but more technical). For a complete presentation, we refer to~\cite{Beirao-Brezzi-Cangiani-Manzini-Marini-Russo2013, Ahmad-Alsaedi-Brezzi-Marini-Russo2013, Beirao-Brezzi-Marini-Russo2013}.

\subsection{VE space}
\label{sec:VE_space}

Let $\T_h$ be a (not necessarily conforming) decomposition of $\Omega$ into $n_P$ non-overlapping (open) polygons $E_\ell$ with flat faces, i. e., $\bar\Omega=\cup_{\ell=1}^{n_P} \bar E_\ell$ with $E_\ell\cap E_{\ell'}=\emptyset$ for $\ell\neq\ell'$. Let $h_E\coloneqq\operatorname{diam}(E)$ and $h\coloneqq\max_{E\in\T_h}h_E$. In the following, we assume that (i) each element $E\in\T_h$ is star-shaped with respect to a ball of radius $\gamma\, h_E$; (ii) the distance between any two vertices of $E$ is larger than $c\, h_E$, for $\gamma, c>0$ constants independent of $h$ and $E$.

Let $k\in\N$ denote the polynomial degree of the method. For any fixed $E\in\T_h$, we introduce the following notations:
\begin{enumerate}[label=(\roman*)]
\item
$\P_k(E)$ is the set of polynomials on $E$ of total degree less or equal to $k$;
\item
$\B(\partial E)\coloneqq\{v\in C^0(\partial E)\ \textrm{s.t. }v|_e\in\P_k(e)\ \textrm{for all edge }e\subset\partial E\}$;
\item 
$\Pi^{\nabla,E}\colon H^1(E)\rightarrow \P_k(E)$ is the energy projection operator defined by
\begin{equation}
	\label{eq:PiNabla}
	a^E(q_k,w-\Pi^{\nabla,E}w)=0\quad\forall\, q_k\in\P_k(E),	
\end{equation}
where $a^E\colon H^1(E)\times H^1(E)\rightarrow\R$ is the local counterpart of the bilinear form $a(\bullet,\bullet)$, namely, $a^E(v,w)=\int_E\nabla v\cdot\nabla w\, \dx$ for all $v,\, w\in H^1(E)$, and $a(v,w)=\sum_{E\in\T_h}a^E(v,w)$ for all $v,\, w\in H^1(\Omega)$. To fix the constant in the definition~\eqref{eq:PiNabla} of $\Pi^{\nabla,E}w$, we further require $\int_E\Pi^{\nabla,E}w\, \dx=0$;
\item
$\Pi^{0,E}\colon L^2(E)\rightarrow \P_k(E)$ is the $L^2$-orthogonal projection operator defined by
\begin{equation}
	\label{eq:Pi0}
	\spE{q_k}{w-\Pi^{0,E}w}=0\quad\forall\, q_k\in\P_k(E).
\end{equation}
There exists a positive constant $C$ such that, for all $u\in H^{s+1}(E)$ with $0\leq s\leq k$, there holds
\begin{gather}
	\label{eq:proj_approx}
	\norm{u-\Pi^{0,E} u}{L^2(E)}\leq C h_E^{s+1} \seminorm{u}{H^{s+1}(E)},
\end{gather}
where $h_E$ is the diameter of the element $E$.
(See \cite{Brenner-Scott2008}).
\end{enumerate}

We can now introduce the (local) enhanced VE space
\begin{equation}
	\label{eq:VEM_space_local}
	W^E_h\coloneqq\left\{ w\in V^E_h\ \textrm{s.t. }\spE{w-\Pi^{\nabla,E}w}{q}=0\ \textrm{for all }q\in\P_k(E)/\P_{k-2}(E)\right\},
\end{equation}
where $V^E_h$ denotes the (local) augmented VE space
\[
V^E_h\coloneqq\left\{w\in H^1(E)\ \textrm{s.t. }w\in\B_k(\partial E)\ \textrm{and }\Delta w\in\P_k(E)\right\},
\]
and $\P_k(E)/\P_{k-2}(E)$ denotes the set of polynomials of total degree $k$ on $E$ that are $L^2$-orthogonal to all polynomials of total degree $k-2$ on $E$ (with the convention $\P_{-1}\coloneqq \emptyset$).  Note, in particular, that $\P_k(E)\subset W^E_h(E)$.
The space $W^E_h$ is equipped with the following set of (local) degrees of freedom (DOFs):
\begin{itemize}
\item
nodal values at all $n_E$ vertices of the polygon $E$;
\item
nodal values at $k-1$ Gauss-Lobatto quadrature points of every edge $e\in\partial E$;
\item 
(for $k\geq 2$) moments up to order $k-2$ in $E$, namely, for $w\in W^E_h$,
\[
\spE{w}{q_{k-2}}\quad\textrm{for all }q_{k-2}\in\P_{k-2}.
\]
\end{itemize}
In particular, $\dim(W^E_h)=n_E k + \frac{k(k-1)}{2}$. It is important to notice that both the energy projection and the $L^2$-orthogonal projection operators are computable only on the basis of degrees of freedom above.

The global enhanced VE space is given by
\begin{equation}
	\label{eq:VEM_space_global}
	W_h\coloneqq\left\{v\in H^1_0(\Omega)\ \textrm{s.t. }v|_E\in W^E_h\ \textrm{for all } E\in\T_h \right\}.
\end{equation}
It is equipped with the following set of (global) DOFs:
\begin{itemize}
	\item
	nodal values at all $n_V$ vertices of $\T_h$;
	\item
	nodal values at $k-1$ Gauss-Lobatto quadrature points of all $n_e$ edges of $\T_h$;
	\item (for $k\geq 2$) moments up to order $k-2$ in all $n_P$ polygons of $\T_h$, namely, for $w\in W_h$,
	\[
	\spE{w}{q_{k-2}}\quad\textrm{for all }q_{k-2}\in\P_{k-2}(E);
	\]
\end{itemize}
and it has dimension $\dim(W_h)=n_V + (k-1)n_e+n_P\frac{(k-1)(k-2)}{2}$. In the following, we set $N_h\coloneqq\dim(W_h)$.

Given a smooth enough function $u$, we define its VE interpolant $u_I$ as the function in $W_h$ verifying, for all $j=1,\ldots, N_h$
\begin{equation}
	\label{eq:VE-interp}
	\textrm{dof}_j(u) = \textrm{dof}_j(u_I), 
\end{equation}
where $\textrm{dof}_j$ is the operator associating its argument to the $j$-th (global) DOF. In~\cite{Beirao-Brezzi-Cangiani-Manzini-Marini-Russo2013} it is shown that there exists a positive constant $C$ such that, for all $E\in \T_h$, there holds
\begin{equation}
	\label{eq:inter-estimate}
	\norm{u-u_I}{L^2(E)}+h\seminorm{u-u_I}{H^1(E)}
	\leq C h^{k+1} \seminorm{u}{H^{k+1}(E)}.
\end{equation}

\subsection{VE bilinear forms}

Based on the classical observation that, given an arbitrary pair of VE functions $v_h,\, w_h\in W^E_h$, the quantities $a^E(v_h,w_h),\, \spE{v_h}{w_h}$ can not be computed, we introduce computable approximations $a^E_h,\, m_h^E\colon W^E_h\times W^E_h\rightarrow\R$, given by
\begin{gather}
	\label{eq:VEM_bilinear_forms}
	\begin{aligned}
		a^E_h(v_h,w_h)&\coloneqq a^E(\Pi^{\nabla,E}v_h,\Pi^{\nabla,E}w_h)
			+ S^E\left((Id-\Pi^{\nabla,E})v_h,(Id-\Pi^{\nabla,E})w_h \right),\\
		m^E_h(v_h,w_h)&\coloneqq\spE{\Pi^{0,E}v_h}{\Pi^{0,E}w_h)}
			+ R^E\left((Id-\Pi^{0,E})v_h,(Id-\Pi^{0,E})w_h \right),
	\end{aligned}
\end{gather}
where $S^E,\, R^E\colon W^E_h\times W^E_h\rightarrow\R$ are symmetric stabilizing bilinear forms fulfilling, for all $v_h\, w_h\in W^E_h$ with $\Pi^{\nabla,E}v_h=0$, $\Pi^{0,E}w_h=0$,
\begin{gather}
	\begin{aligned}
		a^E(w_h,w_h)\lesssim &R^E(w_h,w_h) \lesssim a^E(w_h,w_h),\\
		\spE{w_h}{w_h}\lesssim &S^E(w_h,w_h)\lesssim \spE{w_h}{w_h}.
	\end{aligned}
\end{gather} 
In particular, the local virtual bilinear forms defined in~\eqref{eq:VEM_bilinear_forms} fulfill the $k$-consistency and stability properties, namely, for all $q_h\in \P_k(E)$ and $w_h\in W^E_h$
\begin{gather}
	\label{eq:consistency}
	a^E_h(q_h,w_h)=a^E(q_h,w_h),\quad
	m^E_h(q_h,w_h)=m^E(q_h,w_h),
\end{gather}
and 
\begin{gather}
	\label{eq:stability}
	\begin{aligned}
	a^E(w_h,w_h)\lesssim &a^E_h(w_h,w_h)\lesssim a^E(w_h,w_h),\\
	\spE{w_h}{w_h}\lesssim &m^E_h(w_h,w_h) \lesssim \spE{w_h}{w_h}.
	\end{aligned}
\end{gather}

\begin{remark}
In the numerical experiments, we will approximate the stabilizing forms $S^E(\bullet,\bullet)$, $R^E(\bullet,\bullet)$ with the computable bilinear forms $S^E_h(\bullet,\bullet)$, $R^E_h(\bullet,\bullet)$ defined as follows:
\begin{gather*}
	\begin{aligned}
	S_h^E(v_h,w_h)&\coloneqq
	\sum_{r=1}^{N_E} \textrm{DOF}_r\Big( (\textrm{Id}-\Pi^{\nabla,E})v_h \Big)
	\textrm{DOF}_r\Big( (\textrm{Id}-\Pi^{\nabla,E})w_h \Big),\\
	R_h^E(v_h,w_h)&\coloneqq
	|E|\sum_{r=1}^{N_E} \textrm{DOF}_r\Big( (\textrm{Id}-\Pi^{0,E})v_h \Big)
	\textrm{DOF}_r\Big( (\textrm{Id}-\Pi^{0,E})w_h \Big),
	\end{aligned}
\end{gather*}
where $|E|$ is the area of the polygon $E$, $N_E\coloneqq\dim(W^E_h)$ and $\{\textrm{DOF}_r\}_{r=1}^{N_E}$ denotes the set of local DOFs introduced in Section~\ref{sec:VE_space}.
\end{remark}

The global virtual bilinear forms $a_h,\, m_h\colon W_h\times W_h\rightarrow\R$ are then defined, for all $v_h\, w_h\in W^E_h$, as
\begin{gather*}
	a_h(v_h,w_h)\coloneqq\sum_{E\in\T_h} a^E_h(v_h,w_h),\quad
	m_h(v_h,w_h)\coloneqq\sum_{E\in\T_h} m^E_h(v_h,w_h).
\end{gather*}
From \eqref{eq:stability} it follows that the global virtual bilinear forms are continuous, namely, 
\begin{gather}
	\label{eq:cont_VEM_forms}
	\begin{aligned}
	a_h(v,w)&\lesssim\norm{\nabla v}{L^2(\Omega)}\norm{\nabla w}{L^2(\Omega)},\\
	m_h(v,w)&\lesssim\norm{v}{L^2(\Omega)}\norm{ w}{L^2(\Omega)}.
	\end{aligned}
\end{gather}
We define the discrete $H^1$-seminorm and the discrete $L^2$-norm as follows
\begin{gather}
	\label{eq:VEM-norms}
	\normH{\bullet}^2\coloneqq a_h(\bullet,\bullet),\quad
	\normL{\bullet}^2\coloneqq m_h(\bullet,\bullet).
\end{gather}
Combining \eqref{eq:stability} and \eqref{eq:cont_VEM_forms}, we find that, for all $v_h,w_h\in W_h$, there holds
\begin{gather}
	\label{eq:CS-mh}
	m_h(v,w)\lesssim \normL{v_h}\normL{w_h}.
\end{gather}

\subsection{VE semi-discrete variational problem}

We define the VE approximation to the loading term $f(t)$ for all $t\in (0,T)$ as
\begin{gather*}
	f_h(t)|_E\coloneqq \Pi^{0,E} f(T) \textrm{ for all }E\in \T_h,
\end{gather*}
and the VE approximated initial conditions $u_{0,h},\, z_{0,h}$ as the VE interpolants of $u_0,\, z_0$, specifically, $u_{h,0},\, z_{h,0}$ are piecewise polynomials of degree less than or equal to $k$, with evaluations of DOFs coinciding with those of $u_0,\, z_0$ (see~\eqref{eq:VE-interp}).

The VE semi-discrete approximation to~\eqref{eq:pde_weak} reads: find $u_h\in C^0\left(0,T; W_h\right)\cap C^1\left(0,T;W_h\right)$ such that, for all $v_h\in W_h$ and for a.e. $t\in (0,T)$
\begin{equation}
	\label{eq:pde_weak_sd}
	m_h(u_{h,tt}(t), v_h ) +  \nu m_h(u_{h,t}(t), v_h ) + a_h(u_h(t),v_h) = \sp{f_h(t)}{v_h},
\end{equation}
supplemented with the initial conditions $u_h(\cdot,0)=u_{h,0}$, $u_{h,t}(\cdot,0)=z_{h,0}$.
By classical arguments, it is possible to show that problem \eqref{eq:pde_weak_sd} admits a unique solution $u_h(t)$ (see Section~\ref{thm:VEM_wp}).
Moreover, there holds the following stability result.
\begin{theorem}
\label{thm:VEM_stability}
Let $f_h\in L^2(0,T;L^2(\Omega))$. Then, the unique solution $u_h$ to problem~\eqref{eq:pde_weak_sd} fulfills the following inequality, for all $t\in(0,T)$
\begin{equation}
	\label{eq:VEM_stability}
	\normH{u_h(t)}^2 + \normL{u_{h,t}(t)}^2 
	\lesssim \normH{u_{h,0}}^2 + \normL{z_{h,0}}^2 +\norm{f_h}{L^2(0,t;L^2(\Omega))}^2.
\end{equation}
\end{theorem}

\begin{proof}
Choosing the test function $v_h=u_{h,t}(t)$ in problem~\eqref{eq:pde_weak_sd}, and integrating in time between $0$ and $t$, we find
\begin{gather*}
	\begin{aligned}
	&\int_0^t m_h(u_{h,tt}(s), u_{h,t}(s) )\, \ds 
	+ \nu  \int_0^t m_h(u_{h,t}(s), u_{h,t}(s) )\, \ds 
	+ \int_0^t a_h(u_h(s),u_{h,t}(s) )\, \ds\\
	&\qquad= \int_0^t \sp{f_h(s)}{u_{h,t}(s)}\, \ds.
	\end{aligned}
\end{gather*}
Observe that
\begin{gather*}
    \begin{aligned}
	\int_0^t m_h(u_{h,tt}(s), u_{h,t}(s) )\, \ds 
	&= \frac{1}{2} \int_0^t \frac{d}{ds} m_h(u_{h,t}(s), u_{h,t}(s) ) \, \ds \\
	&= \frac{1}{2} \left( \normL{u_{h,t}(t)}^2 - \normL{z_{h,0}}^2\right),
    \end{aligned}
\end{gather*}
and, analogously,
\begin{gather*}
	\int_0^t a_h(u_{h}(s), u_{h,t}(s) )\, \ds 
	= \frac{1}{2} \int_0^t \frac{d}{ds} a_h(u_{h}(s), u_{h}(s) ) \, \ds 
	= \frac{1}{2} \left( \normH{u_{h}(t)}^2 - \normH{u_{h,0}}^2\right).
\end{gather*}
Moreover, using~\eqref{eq:stability} we find that
\begin{gather*}
	\int_0^t m_h(u_{h,t}(s), u_{h,t}(s) )\, \ds 
	\gtrsim \int_0^t \norm{u_{h,t}(s)}{L^2(\Omega)}^2\, \ds
	\gtrsim\norm{u_{h,t}}{L^2(0,t;L^2(\Omega))}^2.
\end{gather*}
Finally, using the Cauchy-Schwarz and Young's inequalities, we find
\begin{gather*}
	\begin{aligned}
	\int_0^t \sp{f_h(s)}{u_{h,t}(s)}\, \ds
	&\leq \norm{f_h(t)}{L^2(0,t;L^2(\Omega))} \norm{u_{h,t}}{L^2(0,t;L^2(\Omega))}\\
	&\leq \frac{1}{2\varepsilon} \norm{f_h(t)}{L^2(0,t;L^2(\Omega))}^2 
	+ \frac{\varepsilon}{2}\norm{u_{h,t}}{L^2(0,t;L^2(\Omega))}.
	\end{aligned}
\end{gather*}
Hence, \eqref{eq:VEM_stability} follows for $\varepsilon$ sufficiently small. 
\end{proof}

\begin{remark}
\label{rem:VEM_stability_q}
Let the further assumption $u_h\in H^{q}(0,T;H^1_0(\Omega)$ for $q\geq 2$ hold, and denote with $\partial^q_t u_h$ the $q$-th time derivative of $u_h$ (which still fulfills homogeneous Dirichlet boundary conditions on $\partial\Omega$). Then, $w_h\coloneqq \partial^q_t u_h$ satisfies, for all $v_h\in W_h$ and for a.e. $t\in (0,T)$,
\begin{equation*}
	m_h(w_{h,tt}(t), v_h ) +  \nu m_h(w_{h,t}(t), v_h ) + a_h(w_h(t),v_h) = \sp{\partial^q_t f_h(t)}{v_h},
\end{equation*}
coupled with initial conditions $w_h(0)=w_{h,t}(0)=0$.
Theorem~\ref{thm:VEM_stability} then states that
 \begin{equation}
	\label{eq:VEM_stability_q}
	\normH{w_h(t)}^2 + \normL{w_{h,t}(t)}^2 
	 \lesssim \norm{f_h^{(q)}}{L^2(0,t;L^2(\Omega))}^2.
\end{equation}

\end{remark}

\subsection{Error analysis}

To perform the error analysis for the semi-discrete problem, we need to introduce the modified energy projection $\proj^\nabla\colon H^1_0(\Omega)\rightarrow W_h$, where, for any $u\in H^1_0(\Omega)$, $\proj^\nabla u\in W_h$ satisfies, for all $v_h\in W_h$
\begin{gather}
	\label{eq:proj_nabla}
	a_h(\proj^\nabla u,v_h) = a_h(u,v_h).
\end{gather}
Similarly, we define the the modified $L^2$-projection $\proj^0\colon L^2(\Omega)\rightarrow W_h$, where, for any $u\in L^2(\Omega)$, $\proj^0 u\in W_h$ satisfies, for all $v_h\in W_h$
\begin{gather}
	\label{eq:proj_0}
	m_h(\proj^0 u,v_h) = m_h(u,v_h).
\end{gather}
We recall the following approximation results, whose proofs can be found in \cite{Vacca-Beirao2015} and \cite{Vacca2017}, respectively.
\begin{lemma}
\label{lem:proj_nabla_estimate}
For all $u\in H^1_0(\Omega)\cap H^{k+1}(\Omega)$, there holds
\begin{align}
	\label{eq:proj_nabla_estimate1}
	\seminorm{u-\proj^\nabla u}{H^1(\Omega)}&\lesssim  h^k\seminorm{u}{H^{k+1}(\Omega)}, \\
	\label{eq:proj_nabla_estimate2}
	\norm{u-\proj^\nabla u}{L^2(\Omega)}&\lesssim h^{k+1}\seminorm{u}{H^{k+1}(\Omega)}.
\end{align}
\end{lemma}

\begin{lemma}
\label{lem:proj_0_estimate}
For all $u\in H^{k+1}(\Omega)$, there holds
\begin{gather}
	\label{eq:proj_0_estimate}
	\norm{u-\proj^0 u}{L^2(\Omega)}\lesssim h^{k+1}\seminorm{u}{H^{k+1}(\Omega)}.
\end{gather}
\end{lemma}

Let us denote 
$$
\seminorm{\bullet}{L^1(0,t;H^{k+1}(\Omega)}\coloneqq\int_0^t\seminorm{\bullet}{H^{k+1}(\Omega)}\,\ds,\quad 
\seminorm{\bullet}{L^2(0,t;H^{k+1}(\Omega)}\coloneqq\int_0^t\seminorm{\bullet}{H^{k+1}(\Omega)}^2\,\ds.
$$
We are now ready to state the following convergence result, which extends \cite[Theorem 3.3]{Vacca2017} to the case $\nu>0$.

\begin{theorem}
\label{thm:VEM_estimate1}
Let $u$, $u_h$ be the unique solutions of problems \eqref{eq:pde_weak} and \eqref{eq:pde_weak_sd}, respectively. Assume that $u\in C^2\left(0,T; H^1_0(\Omega)\right)$, $u_0,\, z_0\in H^{k+1}(\Omega)$ and $u_t,\, u_{tt},\, f\in L^2\left(0,T;H^{k+1}(\Omega)\right)$, with $k\geq 1$ integer. Then, there holds
\begin{gather}
	\label{eq:VEM_estimate1}
	\begin{aligned}
	&\norm{u_h(t)-u(t)}{H^1(\Omega)} + \norm{u_{h,t}(t)-u_t(t)}{L^2(\Omega)}\\
	& \lesssim h^{k}\Big( 
	\seminorm{u_0}{H^{k+1}(\Omega)}
	+ h \seminorm{z_0}{H^{k+1}(\Omega)}
	+ h \seminorm{f}{L^2(0,t;H^{k+1}(\Omega))}\\
	& + \seminorm{u_{t}}{L^1(0,t;H^{k+1}(\Omega))}
	 + h \seminorm{u_{t}}{L^2(0,t;H^{k+1}(\Omega))} 
	 + h \seminorm{u_{tt}}{L^1(0,t;H^{k+1}(\Omega))} 
	 +  h \seminorm{u_{t}}{L^2(0,t;H^{k+1}(\Omega))} 
	 \Big).
	\end{aligned}
\end{gather}
\end{theorem}

\begin{proof}
The proof follows the same steps as the proof of \cite[Theorem 3.3]{Vacca2017}. We set
\begin{gather*}
	u_h(t)-u(t) = \left( u_h(t) - \proj^\nabla u(t)\right) + \left(\proj^\nabla u(t) - u(t)\right) \eqqcolon \theta(t) + \rho(t).
\end{gather*}
We bound the $\norm{\rho(t)}{H^1(\Omega)}$ by using~\eqref{eq:proj_nabla_estimate1} 
\begin{gather}
	\label{eq:rho_a}
	\begin{aligned}
	\norm{\rho(t)}{H^1(\Omega)} &\lesssim h^k \seminorm{u(t)}{H^{k+1}(\Omega)}
	= h^k \left( \seminorm{u_0}{H^{k+1}(\Omega)} + \int_0^t \seminorm{u_t(s)}{H^{k+1}(\Omega)} \, ds\right)\\
	&= h^k\left( \seminorm{u_0}{H^{k+1}(\Omega)} + \seminorm{u_t}{L^1(0,t;H^{k+1}(\Omega))}\right).
	\end{aligned}
\end{gather}
Similarly, thanks to~\eqref{eq:proj_nabla_estimate2}, we find
\begin{gather}
	\label{eq:rho_b}	
	\norm{\rho_t(t)}{L^2(\Omega)} \lesssim h^{k+1} \seminorm{u_t(t)}{H^{k+1}(\Omega)}
	= h^{k+1}\left( \seminorm{z_0}{H^{k+1}(\Omega)} + \seminorm{u_{tt}}{L^1(0,t;H^{k+1}(\Omega))}\right).
\end{gather}
Note that we can use the estimate~\eqref{eq:proj_nabla_estimate2} since we are assuming $\Omega$ convex.
In order to bound the norm of $\theta(t)$, we note that, for all $v_h\in W_h$ there holds 
\begin{gather*}
	\begin{aligned}
	&m_h(\theta_{tt}(t),v_h) + \nu m_h(\theta_{t}(t),v_h) + a_h(\theta(t),v_h)\\
	&\quad = \sp{f_h(t)-f(t)}{v_h} 
	+ \left[\sp{u_{tt}(t)}{v_h} - m_h(\proj^\nabla u_{tt}(t),v_h)\right]\\
	&\quad\quad+ \nu \left[\sp{u_{t}(t)}{v_h} - m_h(\proj^\nabla u_{t}(t),v_h)\right]\\
	&\quad\eqqcolon \sp{\varphi(t)}{v_h}
	+ \sp{\eta_1(t)}{v_h}
	+ \sp{\eta_2(t)}{v_h},
	\end{aligned}
\end{gather*}
where $\eta_1(t),\,\eta_2(t)\in W_h$ are the Riesz representation of the operators  $\sp{u_{tt}(t)}{\bullet} - m_h(\proj^\nabla u_{tt}(t),\bullet)$ and $\sp{u_{t}(t)}{\bullet} - m_h(\proj^\nabla u_{t}(t),\bullet)$ on the dual space of $W_h$.
Then, $\theta(t)$ is the unique solution of the following weak problem: for all $v_h\in W_h$ there holds
\begin{gather*}
	\left\{\begin{array}{l}
		m_h(\theta_{tt}(t),v_h) + \nu m_h(\theta_{t}(t),v_h) + a_h(\theta(t),v_h)
		= \sp{\varphi(t)+\eta_1(t)+\eta_2(t)}{v_h}\\
		\theta(0) = u_{h,0}-\proj^\nabla u_0\\
		\theta_t(0) = z_{h,0}-\proj^\nabla z_0.
	\end{array}\right.
\end{gather*}
By applying Theorem \ref{thm:VEM_stability} we find:
\begin{equation}
	\label{eq:bound_theta1}
	\begin{aligned}
	&\normH{\theta(t)}^2 + \normL{\theta_t(t)}^2 \\
	&\lesssim \normH{\theta_{0}}^2 + \normL{\theta_t(0)}^2 
	+ \norm{\varphi}{L^2(0,t;L^2(\Omega))}^2 
	+ \norm{\eta_1}{L^2(0,t;L^2(\Omega))}^2 
	+ \norm{\eta_2}{L^2(0,t;L^2(\Omega))}^2 .
	\end{aligned}
\end{equation}
On the other hand, using \eqref{eq:proj_approx}, we find
\begin{gather}
	\label{eq:phi_bound}
	\begin{aligned}
 	\norm{\varphi}{L^2(0,t;L^2(\Omega))}^2
	&=\int_0^t\norm{\varphi(s)}{L^2(\Omega)}^2\,\ds
	=\int_0^t\norm{f_h(s)-f(s)}{L^2(\Omega)}^2\,\ds\\
	&=\int_0^t\sum_{E\in\T_h}\norm{\Pi^{0,E}f(s)-f(s)}{L^2(E)}^2\,\ds\\
	&\lesssim h^{2(k+1)} \seminorm{f}{L^2(0,t;H^{k+1}(\Omega))}^2.
	\end{aligned}
\end{gather}
To bound $\norm{\eta_1}{L^2(0,t;L^2(\Omega))}^2$, we recall \cite[equation (32)]{Vacca2017}: for all $t\in (0,T)$, there holds
\begin{gather*}
	\sp{\eta_1(t)}{v_h}\lesssim h^{k+1} \seminorm{u_{tt}(t)}{H^{k+1}(\Omega)}\norm{v_h}{L^2(\Omega)},
\end{gather*}
yielding
\begin{gather*}
	\norm{\eta_1(t)}{L^2(\Omega)}
	=\sup_{0\neq v_h\in W_h} \frac{\sp{\eta_1(t)}{v_h}}{\norm{v_h}{L^2(\Omega)}}
	\lesssim h^{k+1} \seminorm{u_{tt}(t)}{H^{k+1}(\Omega)}.
\end{gather*}
Hence,
\begin{gather}
	\label{eq:eta1_bound}
	\begin{aligned}
 	\norm{\eta_1}{L^2(0,t;L^2(\Omega))}^2
	&=\int_0^t \norm{\eta_1(s)}{L^2(\Omega)}^2\,\ds
	\lesssim h^{2(k+1)} \seminorm{u_{tt}}{L^2(0,t;H^{k+1}(\Omega))}^2.
	\end{aligned}
\end{gather}
Finally, to bound $\seminorm{\eta_2}{L^2(0,t;L^2(\Omega))}^2$, we observe that, for all $v_h\in W_h$ there holds
\begin{gather*}
	\begin{aligned}
		&\sp{\eta_2(t)}{v_h}
		= \sp{u_t(t)}{v_h} - m_h(\proj^\nabla u_t(t), v_h)\\
		&\quad=\sum_{E\in \T_h}\left[ \spE{u_t(t)}{v_h} - m^E_h(\proj^\nabla u_t(t), v_h)\right]\\
		&\quad=\sum_{E\in \T_h}\left[ \spE{u_t(t)-\Pi^{0,E}u_t(t)}{v_h}
			- m^E_h(\proj^\nabla u_t(t)-\Pi^{0,E}u_t(t), v_h)\right]\\
		&\quad\lesssim \sum_{E\in \T_h}
		\left( \norm{u_t(t)-\Pi^{0,E}u_t(t)}{L^2(E)} + \norm{\proj^\nabla u_t(t)-\Pi^{0,E}u_t(t)}{L^2(E)} \right)
		\norm{v_h}{L^2(E)},
	\end{aligned}
\end{gather*}
where in the third equality we used that $\spE{\Pi^{0,E}u_t(t)}{v_h} = m_h^E(\Pi^{0,E}u_t(t),v_h)$, and the last inequality follows by the Cauchy Schwarz inequality and \eqref{eq:cont_VEM_forms}. Hence, by using \eqref{eq:proj_approx} and applying Lemma \ref{lem:proj_nabla_estimate}, we obtain
\begin{gather*}
	\norm{\eta_2(t)}{L^2(\Omega)}
	=\sup_{0\neq v_h\in W_h} \frac{\sp{\eta_1(t)}{v_h}}{\norm{v_h}{L^2(\Omega)}}
	\lesssim h^{k+1} \seminorm{u_{t}(t)}{H^{k+1}(\Omega)},
\end{gather*}
hence
\begin{gather}
	\label{eq:eta2_bound}
	\seminorm{\eta_2}{L^2(0,t;L^2(\Omega))}^2\lesssim h^{2(k+1)} \seminorm{u_{t}}{L^2(0,t;H^{k+1}(\Omega))}^2.
\end{gather}
The norm of the initial data are derived in \cite[equations (33)-(34)]{Vacca2017}:
\begin{gather}
	\label{eq:init}
	\begin{aligned}
	\normH{\theta_{0}}^2 &\lesssim h^{2k} \seminorm{u_0}{H^{k+1}(\Omega)}^2,\\
	\normL{\theta_t(0)}^2 &\lesssim h^{2(k+1)} \seminorm{z_0}{H^{k+1}(\Omega)}^2.
	\end{aligned}
\end{gather}
Combining \eqref{eq:init}, \eqref{eq:phi_bound}, \eqref{eq:eta1_bound} and \eqref{eq:eta2_bound}, we obtain
\begin{gather}
	\label{eq:theta}
	\begin{aligned}
	&\normH{\theta(t)}^2 + \normL{\theta_t(t)}^2
	\lesssim h^{2k} \Big[\seminorm{u_0}{H^{k+1}(\Omega)}^2 
	+ h^{2} \seminorm{z_0}{H^{k+1}(\Omega)}^2 \\
	&\quad+h^{2}\big(
	 \seminorm{f}{L^2(0,t;H^{k+1}(\Omega))}^2
	+ \seminorm{u_{tt}}{L^2(0,t;H^{k+1}(\Omega))}^2
	+ \seminorm{u_{t}}{L^2(0,t;H^{k+1}(\Omega))}^2\big)\Big].
	\end{aligned}
\end{gather}
Collecting \eqref{eq:rho_a}, \eqref{eq:rho_b} and \eqref{eq:theta}, we conclude 	\eqref{eq:VEM_estimate1}.
\end{proof}

\subsection{Algebraic formulation}
\label{sec:VEM_alg}

Now, we introduce the algebraic formulation of~\eqref{eq:pde_weak_sd} that will be instrumental for the DG discretization in time (see Section~\ref{sec:DG_time}). To this end, we denote with $N_h\colon=\dim(W_h)$, and with $\{\varphi_i\}_{i=1}^{N_h}$ the set of VE basis functions for $W_h$. 
We write, for all $t\in (0,T)$
\begin{gather}
	\label{eq:u_sd}
	u_h(t,x)=\sum_{j=1}^{N_h} U_j(t) \varphi_j(x),
\end{gather}
where $U_j(t)$ is the $j$-th global DOF of $u_h(t)$. Inserting~\eqref{eq:u_sd} into~\eqref{eq:pde_weak_sd} with $v_h=\varphi_i$, we obtain the following system of second-order differential equations
\begin{gather}
	\label{eq:algebraic_form_sd}
	M_h \ddot \U(t) + \nu M_h \dot \U(t) + A_h \U(t) = \F_h(t)
\end{gather}
where 
\begin{itemize}
	\item 
	$\U(t)\coloneqq [U_1(t),\ldots,U_{N_h}(t)]^T\in\R^{N_h}$;
	\item
	$\dot \U(t)\coloneqq [\dot U_1(t),\ldots, \dot U_{N_h}(t)]^T\in\R^{N_h}$ is the vector collecting the DOF of the first temporal derivative of $u$, i.e., $u_{h,t}(t,x)=\sum_{j=1}^{N_h} \dot U_j(t) \varphi_j(x)$;
	\item
	$\ddot \U(t)\coloneqq [\ddot U_1(t),\ldots, \ddot U_{N_h}(t)]^T\in\R^{N_h}$ is the vector collecting the DOF of the second temporal derivative of $u$, i.e., $u_{h,tt}(t,x)=\sum_{j=1}^{N_h} \ddot U_j(t) \varphi_j(x)$;
	\item
	$\F_h(t)\coloneqq[F_1(t),\ldots,F_{N_h}(t)]^T\in\R^{N_h}$ with $F_i(t)\coloneqq\sp{f_h(t)}{\varphi_i}$ for all $i=1,\ldots,N_h$;
	\item
	$M_h,\, A_h\in \R^{N_h\times N_h}$ are the mass and stiffness matrices with elements given as
\begin{gather}
	\label{eq:mass_stiff_VEM}
	\textrm{for all }i,j=1,\ldots,N_h\quad
	(M_h)_{i,j}\coloneqq m_h(\varphi_j,\varphi_i)\quad
	(A_h)_{i,j}\coloneqq a_h(\varphi_j,\varphi_i).
\end{gather}
\end{itemize}
Equation \eqref{eq:algebraic_form_sd} is supplemented with the initial conditions $\U(0)=\U_{h,0}$, $\dot\U(0)=\Z_{h,0}$, where the vector $\U_{h,0}\in\R^{N_h}$ ($\Z_{h,0}\in\R^{N_h}$, respectively) collects the DOFs of $u_{h,0}$ ($z_{h,0}$, respectively).

\section{DG discretization in time}
\label{sec:DG_time}

In this section we first recall the DG (in time) finite dimensional space introduced in \cite{Antonietti-Mazzieri-DalSanto-Quarteroni2017,Antonietti-Mazzieri-Migliorini2020}, and then we apply the DG time integration scheme to \eqref{eq:algebraic_form_sd}. Let the time interval $(0,T]$ be partitioned into $N_T$ time-slabs, i.e., $(0,T]=\cup_{n=1}^{N_T} I_n$, with $I_n\coloneqq(t_{n-1},t_n]$ and $0=t_0<t_1<\cdots<t_n<\cdots<t_{N_T}=T$. We denote with $\tau_n$ the length of the $n$-th time-slab $\tau_n\coloneqq t_n-t_{n-1}$, and we collect the elements of the set $\{\tau_n\}_{n=1}^{N_T}$ in the vector $\Tau$. Moreover, we denote with $\I_{\Tau}$ the partition of the time interval.
Given a sufficiently regular function $v$, we define the time jump operator at $t_n$ for any $n\geq 0$ as
\begin{gather}
	\label{eq:time_jump}
	\jump{v}{n}\coloneqq v(t_n^+)-v(t_n^-),
\end{gather}
where 
\begin{gather*}
	v(t_n^+)=\lim_{\varepsilon\rightarrow 0^+} v(t_n+\varepsilon),\quad
	v(t_n^-)=\lim_{\varepsilon\rightarrow 0^-} v(t_n+\varepsilon).
\end{gather*}
\begin{figure}
    \begin{center}
         \includegraphics[width=0.6\textwidth]{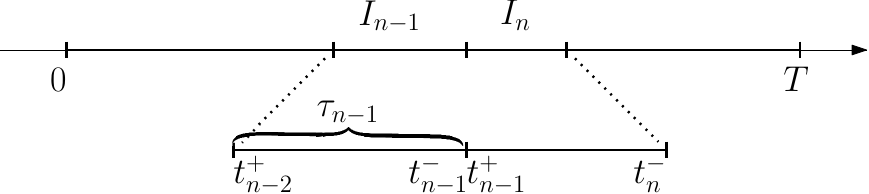}
         \caption{(Top) Example of time partition $\I_\Tau$. (Bottom) Zoom on the time-slabs $I_{n-1}\cup I_n$.}
         \label{fig:time}
    \end{center}
\end{figure}
See Figure \ref{fig:time} for an example of time partition as well as the graphical representation of $t_{n-1}^-,\, t_{n-1}^+$.

Given $r_n\in\N$, we denote the space of polynomials on $I_n$ of degree less than or equal to $r_n$ as $\P_{r_n}(I_n)$, and we define the functional space of piecewise polynomials of degree at least 2 on $\I_{\Tau}$ as
\begin{gather}
	\label{eq:DG_space}
	W_{\Tau}\coloneqq\left\{v\in L^2(0,T) \textrm{ such that } v|_{I_n}\in \P_{r_n}(I_n)\textrm{ with }r_n\geq 2\textrm{ for all }n=1,\ldots, N_T\right\}.
\end{gather}
Since the unknown of \eqref{eq:algebraic_form_sd} is a vector with length $N_h$, we need to introduce the multi-variate version of $W_{\Tau}$. Given the multi-index $\mathbf{r}=(r_1,\ldots,r_{N_\tau})\in\N^{N_T}$, with components $r_n\geq 2$ for all $n=1,\ldots, N_T$, we define
\begin{gather*}
	[W_{\Tau}]^{N_h}\coloneqq\left\{\begin{array}{l}
	\mathbf{V}=(v_1,\ldots,v_{N_h})\in [L^2(0,T)]^{N_h} \colon v_{j}\in W_{\Tau}\ \forall\, j=1,\ldots,N_h
	\end{array}\right\}.
\end{gather*}

Multiplying \eqref{eq:algebraic_form_sd} by a test function $\dot \V\in [W_{\Tau}]^{N_h}$ and integrating on $I_n$, we get
\begin{gather}	
	\label{eq:fd_In}
	\begin{aligned}
		&\spT{M_h \ddot \U}{\dot \V} + \nu \spT{M_h \dot \U}{\dot \V} + \spT{A_h \U}{\dot \V}\\
		&\quad 
		+  M_h \jump{\dot \U}{n}\cdot \dot \V(t^+_n) + A_h \jump{\U}{n}\cdot \V(t^+_n)
		 = \spT{\F_h}{\dot \V}
	\end{aligned}
\end{gather}
where the first two terms in the second row of \eqref{eq:fd_In} are zero since $\U(t)\in C^2(0,T)$, hence they can be added to the equation.
Summing over all time-slabs, we find the following problem: find $\U_\Tau\in [W_{\Tau}]^{N_h}$ such that, for all $\dot\V\in [W_{\Tau}]^{N_h}$ there holds
\begin{gather}
	\label{eq:fd}
	\begin{aligned}
		&\sum_{n=1}^{N_T}\left[
			\spT{M_h \ddot \U_\Tau}{\dot \V} + \nu \spT{M_h \dot \U_\Tau}{\dot \V} + \spT{A_h \U_\Tau}{\dot \V}
		\right]\\
		&\quad + \sum_{n=1}^{N_T-1}\left[
			M_h \jump{\dot \U_\Tau}{n}\cdot \dot \V(t^+_n) + A_h \jump{\U_\Tau}{n}\cdot \V(t^+_n)
		\right]\\
		&\quad + M_h \dot \U_\Tau(0^+)\cdot \dot \V(0^+) + A_h \U_\Tau(0^+)\cdot \V(0^+)\\
		&\quad =
		\sum_{n=1}^{N_T}\left[\spT{\F_h}{\dot \V}\right]
		+ M_h \Z_{h,0}\cdot \dot\V(0^+) + A_h \U_{h,0}\cdot \V(0^+),
	\end{aligned}
\end{gather}
with the initial conditions $\U_\Tau(0)=\U_{h,0}$, $\dot\U_\Tau(0)=\Z_{h,0}$.

Let $\normE{\bullet}\colon [W_{\Tau}]^{N_h}\rightarrow\R$ be defined as
\begin{gather}
	\label{eq:energy_norm}
	\begin{aligned}
	\normE{\V}^2&\coloneqq
	\nu \sum_{n=1}^{N_T}\norm{M_h^{1/2}\dot \V}{L^2(I_n)}^2\\
	 &\quad+ \frac{1}{2}  (M_h^{1/2} \dot \V(0^+))^2  
		+ \frac{1}{2}\sum_{n=1}^{N_T-1} (M_h^{1/2} \jump{\dot \V}{n})^2
		+ \frac{1}{2}  (M_h^{1/2} \dot \V(T^-))^2\\
	&\quad + \frac{1}{2}  (A_h^{1/2} \V(0^+))^2  
		+ \frac{1}{2}\sum_{n=1}^{N_T-1} (A_h^{1/2} \jump{\V}{n})^2
		+ \frac{1}{2}  (A_h^{1/2} \V(T^-))^2.
	\end{aligned}
\end{gather}
In~\cite{Antonietti-Mazzieri-Migliorini2020} it is shown that $\normE{\bullet}$ is a norm on $[W_{\Tau}]^{N_h}$ that, from now on, will be referred to as energy norm. 

Moreover, in \cite[Proposition 3.1]{Antonietti-Mazzieri-DalSanto-Quarteroni2017} it is proved the following stability result: if $\F\in L^2(0,T)$, then the unique solution $\U_{\Tau}\in [W_{\Tau}^{\mathbf r}]^{N_h}$ of \eqref{eq:fd} satisfies  
\begin{gather}
	\label{eq:DG_stability}
	\normE{\U_{\Tau}}\lesssim
	\left( \norm{\F}{L^2(0,T)}^2 + (A_h^{1/2} \U_{h,0})^2 + (M_h^{1/2} \Z_{h,0})^2 \right)^{1/2}.
\end{gather}
In addition, the DG scheme is proved to be convergent, according to the following result (see~\cite[Theorem 3.12]{Antonietti-Mazzieri-DalSanto-Quarteroni2017}).  

\begin{theorem}
\label{thm:DG_error}
If $\U$ is such that $\U|_{I_n}\in \left(H^{q_n}(I_n)\right)^{N_h}$ with $q_n\geq 2$ for all $n=1,\ldots, N_T$, then 
\begin{gather}
	\label{eq:DG_error}
	\normE{\U-\U_{\Tau}}^2\lesssim\sum_{n=1}^{N_T}\frac{\tau_n^{2\beta_n-3}}{r_n^{2q_n-6}}\norm{\U}{(H^{q_n}(I_n))^{N_h}}^2,
\end{gather}
where $\beta_n\coloneqq\min\{r_n+1,q_n\}$ and $r_n\geq 2$ for all $n=1,\ldots, N_T$.
\end{theorem}

\begin{corollary}
\label{cor:err_DG}
Let $\U$ be such that $\U|_{I_n}\in \left(H^{q}(I_n)\right)^{N_h}$ for all $n=1,\ldots, N_T$, with $q\geq 2$. Moreover, let $\tau_n=\Delta t>0$ and $r_n=r\geq 2$ integer, for all $n=1,\ldots,N_T$. Then, the estimate \eqref{eq:DG_error} simplifies as follows:
\begin{gather}
	\label{eq:DG_error2}
	\normE{\U-\U_{\Tau}}^2\lesssim
	\frac{\Delta t^{2\beta-3}}{r^{2q-6}}
	\sum_{n=1}^{N_T}\norm{\U}{(H^{q_n}(I_n))^{N_h}}^2,
\end{gather}
where $\beta\coloneqq\min\{r+1,q\}$. In particular, $\normE{\U-\U_{\Tau}} = O(\Delta t^{\beta-3/2})$ as $\Delta t$ decreases to 0.
\end{corollary}

\section{VEM-DG discretization}
\label{sec:VEM-DG}

In this section we present a tensor product-based space-time discretization of problem \eqref{eq:pde_weak} that combines the VEM presented in Section \ref{sec:VEM_space_discretization} for space discretization, with the DG scheme presented in Section \ref{sec:DG_time} for time integration. The mesh $\Q$ for the space-time domain $\Omega\times (0,T]$ is constructed by tensorizing the polygonal grid $\T_h$ with the time interval partition $\I_{\Tau}$, namely, $\Q\coloneqq\T_h\otimes\I_{\Tau}$. Each element of the space-time mesh $\Q$ is the tensor product of the polygonal mesh $\T_h$ with $I_n$, i.e.,
\begin{gather}
	\Q=\cup_{n=1}^{N_T} Q_n, \textrm{with }Q_n\coloneqq\T_h\otimes I_n \textrm{ for all }n=1,\ldots,N_T.
\end{gather}
\begin{figure}
	\begin{subfigure}[c]{0.45\textwidth}
         \includegraphics[width=\textwidth]{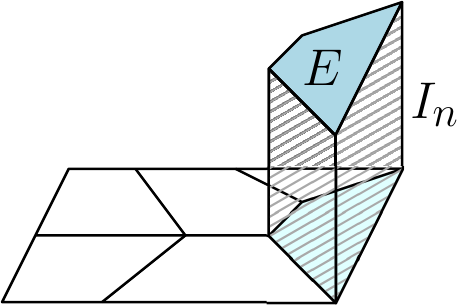}
         \caption{}
     \end{subfigure}
     \hfill
     \begin{subfigure}[c]{0.45\textwidth}
         \centering
         \includegraphics[width=0.8\textwidth]{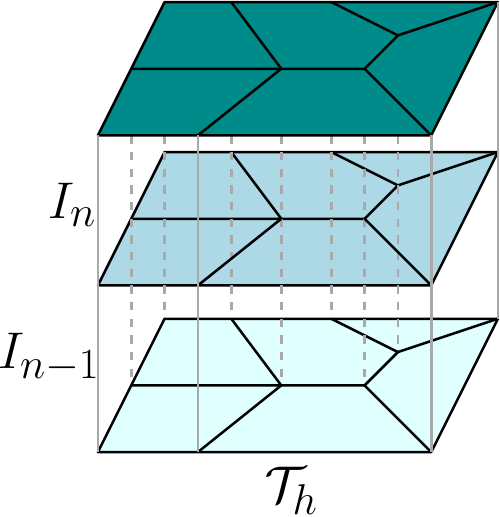}
         \caption{}
     \end{subfigure}
     \caption{(A) Polygon $E\in\T_h$ tensorized with the time-slab $I_n$. (B) Polygonal mesh $\T_h$ tensorized with $I_{n-1}\cup I_n$, namely, $Q_{n-1}\cup Q_n$. Darker color encodes the increasing time instances.}
        \label{fig:tp_mehs}	
\end{figure}
We refer to Figure \ref{fig:tp_mehs} for an example..

The tensor product of the VE space $W_h$ defined in \eqref{eq:VEM_space_global} with the DG space $W_{\Tau}$ defined in \eqref{eq:DG_space} gives the following finite-dimensional space
\begin{gather}
	\W_{h,\Tau} \coloneqq\left\{
	w(x,t)=w_1(x)w_2(t)\colon \T_h\times \I_{\Tau}\rightarrow \R\textrm{ such that }w_1\in W_h \textrm{ and }w_2\in W_{\Tau}
	\right\}.
\end{gather}
Note that, by definition, each $w\in\W_{h,\Tau}$ is continuous in the spatial domain but might be discontinuous in the time domain, i.e., discontinuities are allowed along the interfaces $\T_h\otimes\{t_n\}$, for $n=1,\ldots,N_T-1$. 

To derive the tensor product VEM-DG formulation of the problem of interest, we start from equation~\eqref{eq:pde} in $\Q_n$ multiplied by a test function $\dot w=w_1(x)\dot w_2(t)\in \W_{h,\Tau}$ and we integrate in space and time. Then, we integrate by parts with respect to the space variable and we replace the $L^2(\Omega)$-inner product $\sp{\bullet}{\bullet}$ and the bilinear form $a(\bullet,\bullet)$ with the VE bilinear forms $m_h(\bullet,\bullet)$ and $a_h(\bullet,\bullet)$, respectively. Finally, we add the null terms 
\begin{gather*}
	m_h(\jump{\dot u}{n},\dot w(t^+_n)) + a_h( \jump{u}{n}, \dot w(t^+_n))
\end{gather*}
 and we sum up over all time-slabs. As a result, we get the following problem: find $\ufd\in\W_{h,\Tau}$ such that, for all $w\in \W_{h,\Tau}$ there holds
\begin{gather}
	\label{eq:fd_formulation}
	\A(\ufd,w)=\mathcal F(w),
\end{gather}
where the bilinear form $\A\colon \W_{h,\Tau}\times \W_{h,\Tau}\rightarrow\R$ and the linear form $\mathcal F\colon\W_{h,\Tau}\rightarrow\R$  are respectively given by
\begin{gather*}
	\begin{aligned}
		\A(v,w)&\coloneqq
		\sum_{n=1}^{N_T}\Big[
			m_h(v_1,w_1) \spT{\ddot v_2}{\dot w_2} 
			+ \nu m_h(v_1,w_1) \spT{\dot v_2}{\dot w_2}\\
            &\quad \quad+ a_h(v_1,w_1) \spT{v_2}{\dot w_2}
		\Big]\\
		&\quad + \sum_{n=1}^{N_T-1}\Big[
			m_h(v_1,w_1)\jump{\dot v_2}{n} \dot w_2(t^+_n) +a_h(v_1,w_1) \jump{v_2}{n} w_2(t^+_n)
		\Big]\\
		&\quad+ m_h(v_1,w_1) \dot v_2(0^+) \dot w_2(0^+) + a_h(v_1,w_1) v_2(0^+)w_2(0^+),
	\end{aligned}
\end{gather*}
and
 \begin{gather*}
 	\begin{aligned}
		\mathcal F(w)&\coloneqq
		\sum_{n=1}^{N_T}(f_h,w)_{L^2(\Omega\times I_n)} + m_h(z_{h,0},w_1)\dot w_2(0^+) + a_h(u_{h,0},w_1)w_2(0^+),
	\end{aligned}
 \end{gather*}
for any $v(x,t)=v_1(x)v_2(t)$ and $w(x,t)=w_1(x)w_2(t)$.

There holds the following results.
\begin{lemma}
\label{lem:triple_norm}
The function $\vertiii{\bullet}\colon H^2(0,T;H^1_0(\Omega))\rightarrow\R$ defined as
\begin{gather}
	\label{eq:triple_norm}
	\begin{aligned}
		\vertiii{w}^2&\coloneqq \nu \sum_{n=1}^{N_T}\int_{I_n}\normL{\dot w}^2\, \dt\\
		&\quad+\frac{1}{2}\normL{\dot w(0^+,\cdot)}^2 +\frac{1}{2} \sum_{n=1}^{N_T-1}\normL{\jump{\dot w}{n}}^2 +\frac{1}{2}\normL{\dot w(T^-,\cdot)}^2\\
		&\quad+\frac{1}{2}\normH{w(0^+,\cdot)}^2 +\frac{1}{2}\sum_{n=1}^{N_T-1}\normH{\jump{w}{n}}^2 +\frac{1}{2}\normH{w(T^-,\cdot)}^2
	\end{aligned}
\end{gather}
is a norm on $H^2(0,T;H^1_0(\Omega))$.
\end{lemma}

\begin{proof}
It is clear that the function $\vertiii{\bullet}$ satisfies the homogeneity and subadditivity properties. Moreover, if $w=0$, then it immediately follows that $\vertiii{w}=0$. Therefore, $\vertiii{\bullet}$ is a seminorm on $H^2(0,T;H^1_0(\Omega))$. We show that $\vertiii{w}=0$ implies $w=0$ following the same steps as in the proof of \cite[Proposition 2]{Antonietti-Mazzieri-Migliorini2020}. 

The fact that $\vertiii{w}=0$ implies that all the terms at the right-hand side of \eqref{eq:triple_norm} are zero. In particular, for all $n=1,\ldots, N_T$, there holds
\begin{gather*}
	\int_{I_n}\normL{\dot w}^2\, \dt =0,
\end{gather*}
which, in turns, implies $\dot w\equiv 0$ in $\Omega\times I_n$, that is, $w\equiv C_n$ in $\Omega\times I_n$ for $\{C_n\}_{n=1}^{N_T}$ a collection of constants. 
For $n=1$, we have $w\equiv C_1$ in $\Omega\times I_1$. In addition, from
\begin{gather*}
	\normL{w(0^+,\cdot)}=0,
\end{gather*}
we get $w(0^+,\cdot)\equiv 0$ in $\Omega$. Hence, we conclude $C_1=0$, i.e., $w\equiv 0$ in $\Omega\times I_1$.

We proceed now by induction, namely, we assume $w\equiv 0$ in all $\Omega\times I_m$ for $m\leq n-1$, and we show that $w\equiv 0$ in $\Omega\times I_n$. From
\begin{gather*}
	\normH{\jump{w}{n-1}}=0
\end{gather*}
we get $\jump{w}{n-1}=0$, i.e., $w(t_{n-1}^+,x)=w(t_{n-1}^-,x)$ for a.e. $x\in\Omega$. Since $w(t_{n-1}^-,\cdot)\equiv 0$ in $\Omega$ by assumption, we get $w(t_{n-1}^+,\cdot)\equiv 0$ in $\Omega$, which in turns implies $C_n=0$.
\end{proof}

\begin{lemma}
\label{lem:coercivity}
For all $w(x,t)=w_1(x)w_2(t)\in \W_{h,\Tau}$ there holds
\begin{gather}
	\label{eq:coercivity}
	\vertiii{w}^2=\A(w,w).
\end{gather}
\end{lemma}

\begin{proof}
Given $w(x,t)=w_1(x)w_2(t)\in \W_{h,\Tau}$, we have
\begin{gather}
	\label{eq:starting}
	\begin{aligned}
		\A(w,w)&\coloneqq
		\sum_{n=1}^{N_T}\left[
			\normL{w_1}^2 \spT{\ddot w_2}{\dot w_2} 
			+\nu \normL{w_1}^2 \norm{\dot w_2}{L^2(I_n)}^2 + \normH{w_1}^2 \spT{w_2}{\dot w_2}
		\right]\\
		&\quad + \sum_{n=1}^{N_T-1}\left[
			\normL{w_1}^2\jump{\dot w_2}{n} \dot w_2(t^+_n) + \normH{w_1}^2  \jump{v_2}{n} w_2(t^+_n)
		\right]\\
		&\quad+\normL{w_1}^2 \left(\dot w_2(0^+)\right)^2 + \normH{w_1}^2 \left(w_2(0^+)\right)^2.
	\end{aligned}
\end{gather}
Integrating by parts, we get
\begin{gather*}
	\spT{\ddot w_2}{\dot w_2} = -\spT{\dot w_2}{\ddot w_2} + \left( \dot w_2(t_n^- \right)^2 - \left( \dot w_2(t_{n-1}^+ \right)^2,
\end{gather*}
which implies 
\begin{gather}
	\label{eq:parta}
	\spT{\ddot w_2}{\dot w_2} = \frac{1}{2}\left( \dot w_2(t_n^- \right)^2 - \frac{1}{2}\left( \dot w_2(t_{n-1}^+ \right)^2.
\end{gather}
Analogously, we derive
\begin{gather}
	\label{eq:partb}
	\spT{\dot w_2}{w_2} = \frac{1}{2}\left( w_2(t_n^- \right)^2 - \frac{1}{2}\left( w_2(t_{n-1}^+ \right)^2.
\end{gather}
Inserting \eqref{eq:parta} and \eqref{eq:partb} into \eqref{eq:starting}, and performing simple computations, we derive
\begin{gather*}
	\begin{aligned}
		\A(w,w)&= 
		\nu \normL{w_1}^2\sum_{n=1}^{N_T}\norm{\dot w_2}{L^2(I_n)}^2\\
		&\quad+\frac{1}{2} \normL{w_1}^2 \left[\left(\dot w_2(0^+)\right)^2 + \sum_{n=1}^{N_T-1}\left(\jump{\dot w_2}{n}\right)^2 
			+ \left(\dot w_2(T^-)\right)^2\right]\\
		&\quad+\frac{1}{2}\normH{w_1}^2 \left[\left(w_2(0^+)\right)^2 + \sum_{n=1}^{N_T-1}\left(\jump{ w_2}{n}\right)^2 
			+ \left(w_2(T^-)\right)^2\right]
	\end{aligned}
\end{gather*}
and we conclude by observing that 
\[
	\normL{w_1}^2\sum_{n=1}^{N_T}\norm{\dot w_2}{L^2(I_n)}^2 
	= \sum_{n=1}^{N_T}\int_{I_n}m_h(\dot w,\dot w)\, \dt
	= \sum_{n=1}^{N_T}\int_{I_n}\normL{\dot w}^2\, \dt,
\]
and 
\begin{gather*}
	\begin{aligned}
		\normL{w_1}^2 \left(\dot w_2(0^+)\right)^2 &= \normL{\dot w(0^+,\cdot)}^2 \\
		\normL{w_1}^2 \left(\jump{\dot w_2}{n}\right)^2 &=\normL{\jump{\dot w}{n}}^2\\
		\normL{w_1}^2 \left(\dot w_2(T^-)\right)^2 &= \normL{\dot w(T^-,\cdot)}^2
	\end{aligned}
\end{gather*}
as well as
\begin{gather*}
	\begin{aligned}
		\normH{w_1}^2 \left(w_2(0^+)\right)^2 &= \normL{w(0^+,\cdot)}^2 \\
		\normH{w_1}^2 \left(\jump{ w_2}{n}\right)^2 &=\normL{\jump{w}{n}}^2\\
		\normH{w_1}^2 \left( w_2(T^-)\right)^2 &= \normL{w(T^-,\cdot)}^2.
	\end{aligned}
\end{gather*}
\end{proof}

\begin{theorem}[Well-posedness]
There exists a unique solution to the VEM-DG problem~\eqref{eq:fd_formulation}.
\end{theorem}

\begin{proof}
Lemma~\ref{lem:coercivity} implies that the bilinear form $\A(\bullet,\bullet)$ is coercive, with coercivity constant 1. The continuity of $\mathcal F$ follows from Cauchy-Schwarz inequality and the continuity of the global virtual bilinear forms \eqref{eq:cont_VEM_forms}.
\end{proof}

\subsection{Algebraic formulation}

In this section we derive the algebraic formulation of the fully discrete problem \eqref{eq:fd_formulation}. We start noticing that the use of DG in time allow us to compute the discrete solution separately, one time-slab at a time. In particular, given $1\leq n\leq N_T$, problem \eqref{eq:fd_formulation} restricted to $I_n$ reads: find $\ufd^n\coloneqq \ufd|_{I_n}\in W_{h}\otimes \P_{r_n}(I_n)$ such that, for all $w\in W_{h}\otimes \P_{r_n}(I_n)$ there holds
\begin{gather}
	\label{eq:formulation_In}
	\A_n(\ufd^n,w)=\mathcal F_n(w),
\end{gather}
where 
\begin{gather*}
	\begin{aligned}
		\A_n(v,w)&\coloneqq
			m_h(v_1,w_1) \spT{\ddot v_2}{\dot w_2} + \nu m_h(v_1,w_1) \spT{\dot v_2}{\dot w_2}\\
            &\quad + a_h(v_1,w_1) \spT{v_2}{\dot w_2}
		  + m_h(v_1,w_1)\dot v_2(t_{n-1}^+) \dot w_2(t_{n-1}^+) \\
		&\quad + a_h(v_1,w_1) v_2(t_{n-1}^+) w_2(t_{n-1}^+), 
	\end{aligned}
\end{gather*}
and
 \begin{gather*}
 	\begin{aligned}
		\mathcal F_n(w)&\coloneqq
		(f_h,w)_{L^2(\Omega\times I_n)} 
		+ m_h(\dot u_{h,\Tau}^{n-1}(t_{n-1}^-,\cdot),\dot w(t_{n-1}^-,\cdot)
		+ a_h(\ufd^{n-1}(t_{n-1}^-,\cdot),w(t_{n-1}^-,\cdot),
	\end{aligned}
 \end{gather*}
for any $v(x,t)=v_1(x)v_2(t)\in W_{h}\otimes \P_{r_n}(I_n)$ and $w(x,t)=w_1(x)w_2(t)\in W_{h}\otimes \P_{r_n}(I_n)$. Note, in particular, that the solution computed for $I_{n-1}$ is used as initial condition for the current time-slab. 

Following the same notation as in Section \ref{sec:VEM_alg}, we write $W_h=span\{\varphi_{j}\}_{j=1}^{N_h}$. Moreover, we denote with $\{\psi_m\}_{m=1}^{r_n}$ a basis for $\P_{r_n}(I_n)$. Then, the trial function $\ufd^n$ can be expressed as linear combination of the tensor product basis function $\{\varphi_j\phi_m,\, j=1,\ldots,N_h,\, m=1,\ldots,r_n+1\}$, namely
\begin{gather}
	\label{eq:udf_n}
	\ufd^n(x,t)=\sum_{j=1}^{N_h}\sum_{m=1}^{r_n+1} \alpha^n_{j,m} \varphi_j(x) \psi_m(t),
\end{gather}
where $\alpha^n_{j,m}\in\R$ for all $j=1,\ldots,N_h,\, m=1,\ldots,r_n+1$. Inserting \eqref{eq:udf_n} into \eqref{eq:formulation_In} and taking $w(x,t)=\varphi_i(x)\psi_\ell(t)$, we get
\begin{gather*}
	A^n \boldsymbol\alpha^n = \mathbf{F}^n,
\end{gather*}
where 
\begin{itemize}
	\item 
	$\boldsymbol\alpha^n\in\R^{N_h(r_n+1)}$ is the solution vector;
	\item
	$A^n\in \R^{N_h(r_n+1)\times N_h(r_n+1)}$ has the following structure:
	\begin{gather*}
	A^n = M_h\otimes (N_1 + \nu N_2 + N_4) + A_h\otimes (N_3 + N_5), 
	\end{gather*}
	where $M_h,\, A_h\in\R^{N_h\times N_h}$ are the mass and stiffness matrices defined in \eqref{eq:mass_stiff_VEM}, and $N_1,\, N_2,\, N_3,\, N_4,\, N_5\in\R^{(r_n+1)\times (r_n+1)}$ are defined as
	\begin{gather*}
		\begin{aligned}
			(N_1)_{\ell,m}= (\ddot\psi_m,\dot\psi_\ell)_{L^2(I_n)},\quad (N_2)_{\ell,m}= (\dot\psi_m,\dot\psi_\ell)_{L^2(I_n)},
			\quad (N_3)_{\ell,m}= (\psi_m,\dot\psi_\ell)_{L^2(I_n)},\\
			(N_4)_{\ell,m}= \dot\psi_m(t_{n-1}^+)\dot\psi_\ell(t_{n-1}^+),\quad 
			(N_5)_{\ell,m}=\psi_m(t_{n-1}^+)\psi_\ell(t_{n-1}^+);
		\end{aligned}
	\end{gather*}
	\item 
	$\mathbf{F}^n\in \R^{N_h(r_n+1)}$ is the known vector with elements
	\begin{gather*}
	(\mathbf{F}^n)_{i,\ell}= (f_h,\varphi_i\psi_\ell)_{L^2(\Omega\times I_n)} 
		+ (M_h \otimes N_6 ) \boldsymbol\alpha^{n-1}
		+ (A_h \otimes N_7) \boldsymbol\alpha^{n-1},
	\end{gather*}
	where $N_6,\, N_7\in\R^{(r_n+1)\times (r_n+1)}$ are defined as
	\begin{gather*}
		\begin{aligned}
			(N_6)_{\ell,m}=\dot\psi_m(t_{n-1}^-)\dot\psi_\ell(t_{n-1}^-),
			\quad (N_7)_{\ell,m}=\psi_m(t_{n-1}^-)\psi_\ell(t_{n-1}^-).
		\end{aligned}
		\end{gather*}
\end{itemize}

\section{Error analysis}
\label{sec:error_analysis}

Before stating the convergence result for the tensor product VEM-DG method, we introduce the following auxiliary lemma.

\begin{lemma}
\label{lem:norm_equivalence}
Let $\ufd\in\W_{h,\Tau}$ and $\U\in [W_{\Tau}]^{N_h}$ be the solutions of problems~\eqref{eq:fd_formulation} and~\eqref{eq:fd}, respectively. Then, 
$$\vertiii{\ufd}=\normE{\U}.$$
\end{lemma}

\begin{proof}
We follow the same reasoning as in~\cite[Proposition 3]{Antonietti-Mazzieri-Migliorini2020}. We write $\ufd(x,t)=u_1(x)u_2(t)$, with $u_2\in W_{\Tau}$ and $u_1(x)=\sum_{j=1}^{N_h}U_j\varphi_j(x)$, $\{\varphi_j\}_{j=1}^{N_h}$ being the VE basis functions. We set $\U(t)=[U_1,\ldots,U_{N_h}]^Tw_2(t)\in[W_{\Tau}]^{N_h}$. By definition~\eqref{eq:triple_norm}, we have
\begin{gather}
	\label{eq:norm_equiv_p1}
	\begin{aligned}
		\vertiii{\ufd}^2
		&=\nu \sum_{n=1}^{N_T}\int_{I_n}\normL{\dot w}^2\, \dt\\
		&\quad+\frac{1}{2}\normL{\dot \ufd(0^+,\cdot)}^2 +\frac{1}{2} \sum_{n=1}^{N_T-1}\normL{\jump{\dot \ufd}{n}}^2 +\frac{1}{2}\normL{\dot \ufd(T^-,\cdot)}^2\\
		&\quad+\frac{1}{2}\normH{\ufd(0^+,\cdot)}^2 +\frac{1}{2}\sum_{n=1}^{N_T-1}\normH{\jump{\ufd}{n}}^2 +\frac{1}{2}\normH{\ufd(T^-,\cdot)}^2.
	\end{aligned}
\end{gather}
We observe that
\begin{gather*}
	\begin{aligned}
		\int_{I_n}\normL{\dot w}^2\, \dt 
		&= \int_{I_n} \sum_{i,j=1}^{N_h} \left( U_i u_2(t) \right)(M_h)_{i,j} \left( U_j u_2(t) \right)\, \dt\\
		&=\int_{I_n} \U(t)^T M_h \U(t)\, \dt = \norm{M_h^{1/2}\U}{L^2(I_n)}^2.
	\end{aligned}
\end{gather*}
Hence, $\displaystyle{\sum_{n=1}^{N_T}\int_{I_n}\normL{\dot w}^2\, \dt = \sum_{n=1}^{N_T} \norm{M_h^{1/2}\U}{L^2(I_n)}^2}$.
We now focus on the terms in the second line of~\eqref{eq:norm_equiv_p1}. We have
\begin{gather*}
	\begin{aligned}
	\normL{\dot \ufd(0^+,\cdot)}^2 &= \normL{u_1}^2(\dot u_2(0^+))^2
	= \sum_{i,j=1}^{N_h} \left(U_i \dot u_2(0^+)\right) (M_h)_{i,j}\left(U_i \dot u_2(0^+)\right)\\
	&= \left(M_h^{1/2}\dot \U(0^+)\right)^2,
	\end{aligned}
\end{gather*}
and, similarly, we find $\normL{\dot \ufd(T^-,\cdot)}^2 = \left(M_h^{1/2}\dot \U(T^-)\right)^2$. Moreover,
\begin{gather*}
	\begin{aligned}
	\normL{\jump{\dot \ufd}{n}}^2 &= \normL{u_1}^2(\jump{\dot u_2}{n})^2
	= \sum_{i,j=1}^{N_h} \left(U_i\jump{\dot u_2}{n}\right) (M_h)_{i,j}\left(U_i \jump{\dot u_2}{n}\right)
	= \left(M_h^{1/2}\jump{\dot \U}{n}\right)^2.
	\end{aligned}
\end{gather*}
We conclude observing that the terms in the third line of~\eqref{eq:norm_equiv_p1} can be treated analogously. 
\end{proof}

\begin{remark}
\label{rem:norm_equiv}
Lemma~\ref{lem:norm_equivalence} extends to $e_h\coloneqq u_h-\ufd$, $u_h$ being the solution to the semi-discrete problem~\eqref{eq:pde_weak_sd}.
\end{remark}

\begin{theorem}[Error estimate]
\label{thm:error_conv}
Let the assumptions of Theorem \ref{thm:VEM_estimate1} and Theorem \ref{thm:DG_error} hold.
Then, there holds
\begin{equation}
	\label{eq:error}
	\begin{aligned}
	\vertiii{u-\ufd}
	&\lesssim T
	\Bigg[\sum_{n=1}^{N_T}\frac{\tau_n^{2\beta_n-3}}{r_n^{2q_n-6}}
	\left(\normH{u_{h,0}}^2 + \normL{z_{h,0}}^2 + |f_h|^2_{H^{q_n}(0,t;L^2(\Omega))} \right)\Bigg]^{1/2}\\
	& +  T\, h^{k}
	\Bigg[
	\seminorm{u_0}{H^{k+1}(\Omega)}^2 
	+ h^2 \seminorm{z_0}{H^{k+1}(\Omega)}^2 
	+ h^2 \seminorm{f}{L^2(0,T;H^{k+1}(\Omega))}^2\\
	&\quad\quad + \seminorm{u_{t}}{L^1(0,T;H^{k+1}(\Omega))}^2 
	 + h^2 \seminorm{u_{t}}{L^2(0,T;H^{k+1}(\Omega))}^2 \\
	 &\quad\quad+ h^2 \seminorm{u_{tt}}{L^1(0,T;H^{k+1}(\Omega))}^2 
	 +  h^2 \seminorm{u_{tt}}{L^2(0,T;H^{k+1}(\Omega))}^2 
	\Bigg]^{1/2}.
	\end{aligned}
\end{equation}
\end{theorem}

\begin{proof}
Let $u_h(t)\in C^0\left(0,T; W_h\right)\cap C^1\left(0,T;W_h\right)$ be the solution to the semi-discrete problem~\eqref{eq:pde_weak_sd}. Then, we split the error $e\coloneqq u-\ufd=(u-u_h)+(u_h-\ufd)$, where $e_h\coloneqq u-u_h$ is the error due to the space approximation by means of the VEM, and $e_\Tau\coloneqq u_h-\ufd$ is the error due to the DG time discretization. By triangular inequality, we have
\begin{gather}
	\label{eq:error_p1}
	\vertiii{u-\ufd} \leq \vertiii{e_h} + \vertiii{e_\Tau}.
\end{gather}
We start bounding the second contribution to the norm of the error. Applying Lemma~\ref{lem:norm_equivalence} (and Remark~\ref{rem:norm_equiv}) and Theorem~\ref{thm:DG_error}, we find
\begin{gather*}
	\vertiii{e_\Tau}^2 
	\lesssim \sum_{n=1}^{N_T}\frac{\tau_n^{2\beta_n-3}}{r_n^{2q_n-6}}\norm{\U}{(H^{q_n}(I_n))^{N_h}}^2,
\end{gather*}
and by definition of the $H^{q_n}(I_n)$-norm we get
\begin{gather*}
	\vertiii{e_\Tau}^2 
	\lesssim \sum_{n=1}^{N_T}\frac{\tau_n^{2\beta_n-3}}{r_n^{2q_n-6}}
	\int_{I_n}\left( \norm{u_h(t)}{L^2(\Omega)}^2 + \cdots + \norm{\partial_t^{q_n} u_h(t)}{L^2(\Omega)}^2\right)\, \dt,
\end{gather*}
where $\partial_t^{q_n} u_h$ is the $q_n$-th time derivative of $u_h$.
Applying the Poincar\'e inequality, \eqref{eq:stability} and Theorem \ref{thm:VEM_stability}, we find
\begin{gather*}
	\begin{aligned}
	\norm{u_h(t)}{L^2(\Omega)}^2
	&\lesssim \seminorm{u_h(t)}{H^1(\Omega)}^2
	\lesssim \normH{u_h(t)}\\
	&\lesssim \left( \normH{u_{h,0}}^2 + \normL{z_{h,0}}^2 +\norm{f_h}{L^2(0,t,L^2(\Omega))}^2 \right),
	\end{aligned}
\end{gather*}
so that
\begin{gather*}
	\int_{I_n}\norm{u_h(t)}{L^2(\Omega)}^2\, \dt 
	\lesssim T \left( \normH{u_{h,0}}^2 + \normL{z_{h,0}}^2 +\norm{f_h}{L^2(0,T,L^2(\Omega))}^2 \right).
\end{gather*}
Similarly, applying the Poincar\'e inequality, \eqref{eq:stability} and Remark~\ref{rem:VEM_stability_q}, for all $1\leq \alpha\leq q_n$ integer, we obtain
\begin{gather*}
	\int_{I_n}\norm{\partial_t^{q_n} u_h(t)}{L^2(\Omega)}^2\, \dt 
	\lesssim  T\norm{\partial_t^{q_n} f_h}{L^2(0,T,L^2(\Omega))}^2.
\end{gather*}
Hence, we have shown that
\begin{gather}
	\label{eq:t_err}
	\vertiii{e_\Tau}^2
	\lesssim T \sum_{n=1}^{N_T}\frac{\tau_n^{2\beta_n-3}}{r_n^{2q_n-6}}
	\left(\normH{u_{h,0}}^2 + \normL{z_{h,0}}^2 + |f_h|^2_{H^{q_n}(0,T;L^2(\Omega))} \right).
\end{gather}

We consider now the error due to the VEM approximation in space. Recalling that $e_h\in C^1(0,T;W_h)$, we have
\begin{gather}
	\label{eq:s_err0}
	\begin{aligned}
		\vertiii{e_h}^2
			&= \sum_{n=1}^{N_T}\int_{I_n}\normL{\dot e_h}^2\, \dt
			+\frac{1}{2}\normL{\dot e_h(0^+)}^2 +\frac{1}{2}\normL{\dot e_h(T^-)}^2\\
		&\quad+\frac{1}{2}\normH{e_h(0^+)}^2 +\frac{1}{2}\normH{e_h(T^-)}^2.
	\end{aligned}
\end{gather}
Since $u_{h,0}$ is the interpolant of degree $k$ of $u_0$, thanks to~\eqref{eq:inter-estimate} there holds
\begin{gather}
	\label{eq:s_err1}
	\normH{e_h(0^+)}=\normH{u_0-u_{h,0}}\lesssim h^k \seminorm{u_0}{H^{k+1}(\Omega)},
\end{gather}
and similarly, since $z_{h,0}$ is the interpolant of degree $k$ of $z_0$
\begin{gather}
	\label{eq:s_err2}
	\normL{\dot e_h(0^+)}=\normL{z_0-z_{h,0}}\lesssim h^{k+1} \seminorm{z_0}{H^{k+1}(\Omega)}.
\end{gather}
Using \eqref{eq:stability} and Theorem \ref{thm:VEM_estimate1}, we find:
\begin{gather}
\label{eq:s_err3}
	\begin{aligned}
	&\normH{e_h(T^-)}^2 + \normL{\dot e_h(T^-)}^2
	= \normH{u(T^-)-u_{h}(T^-)}^2 + \normL{u_t(T^-)-u_{h,t}(T^-)}^2\\
	&\quad\lesssim \norm{u(T^-)-u_{h}(T^-)}{H^1(\Omega)}^2 +  \norm{u_t(T^-)-u_{h,t}(T^-)}{L^2(\Omega)}^2\\
	&\quad\lesssim h^{2k}\Big( 
	\seminorm{u_0}{H^{k+1}(\Omega)}^2 
	+ h^2 \seminorm{z_0}{H^{k+1}(\Omega)}^2 
	+ h^2 \seminorm{f}{L^2(0,T;H^{k+1}(\Omega))}^2\\
	&\quad\quad + \seminorm{u_{t}}{L^1(0,T;H^{k+1}(\Omega))}^2 
	 + h^2 \seminorm{u_{t}}{L^2(0,T;H^{k+1}(\Omega))}^2 
	 + h^2 \seminorm{u_{tt}}{L^1(0,T;H^{k+1}(\Omega))}^2 
	 +  h^2 \seminorm{u_{tt}}{L^2(0,T;H^{k+1}(\Omega))}^2 
	 \Big).
	\end{aligned}
\end{gather}

Finally, thanks to \eqref{eq:stability} and Theorem \ref{thm:VEM_estimate1}, we find:
\begin{gather}
	\label{eq:s_err5}
	\begin{aligned}
	&\int_{I_n}\normL{\dot e_h(s)}^2\, \ds
	=\int_{I_n} \normL{u_t(s)-u_{h,t}(s)}^2\, \ds
	\lesssim \int_{I_n} \norm{u_t(s)-u_{h,t}(s)}{L^2(\Omega)}^2\, \ds\\
	&\quad \lesssim \tau_n h^{2k}\Big( 
	\seminorm{u_0}{H^{k+1}(\Omega)}^2 
	+ h^2 \seminorm{z_0}{H^{k+1}(\Omega)}^2 
	+ h^2 \seminorm{f}{L^2(0,T;H^{k+1}(\Omega))}^2\\
	&\quad\quad + \seminorm{u_{t}}{L^1(0,T;H^{k+1}(\Omega))}^2 
	 + h^2 \seminorm{u_{t}}{L^2(0,T;H^{k+1}(\Omega))}^2 
	 + h^2 \seminorm{u_{tt}}{L^1(0,T;H^{k+1}(\Omega))}^2 
	 +  h^2 \seminorm{u_{tt}}{L^2(0,T;H^{k+1}(\Omega))}^2 
	 \Big).
	\end{aligned}
\end{gather}
Inserting \eqref{eq:s_err1}, \eqref{eq:s_err2}, \eqref{eq:s_err3} and \eqref{eq:s_err5} into \eqref{eq:s_err0}, and using that $\sum_{n=1}^{N_T}\tau_n=T$, we obtain
\begin{gather}
	\label{eq:s_err}
	\begin{aligned}
	&\vertiii{e_h}^2 \lesssim T h^{2k}\Big( 
	\seminorm{u_0}{H^{k+1}(\Omega)}^2 
	+ h^2 \seminorm{z_0}{H^{k+1}(\Omega)}^2 
	+ h^2 \seminorm{f}{L^2(0,T;H^{k+1}(\Omega))}^2\\
	&\quad\quad + \seminorm{u_{t}}{L^1(0,T;H^{k+1}(\Omega))}^2 
	 + h^2 \seminorm{u_{t}}{L^2(0,T;H^{k+1}(\Omega))}^2 
	 + h^2 \seminorm{u_{tt}}{L^1(0,T;H^{k+1}(\Omega))}^2 
	 +  h^2 \seminorm{u_{tt}}{L^2(0,T;H^{k+1}(\Omega))}^2 
	 \Big).
	 \end{aligned}
\end{gather}
The final result \eqref{eq:error} follows from \eqref{eq:t_err} and \eqref{eq:s_err}.
\end{proof}

\begin{corollary}
\label{cor:err_DG_VEM}
Let $u\in C^2\left(0,T; H^1_0(\Omega)\right)$, $u_0,\, z_0\in H^{k+1}(\Omega)$ and $u_t,\, u_{tt},\, f\in L^2\left(0,T;H^{k+1}(\Omega)\right)$, with $k\geq 1$ integer. Moreover, let $u \in H^{q}\left (I_n; H^1_0(\Omega)\right)$ for all $n=1,\ldots, N_T$, with $q\geq 2$, with $\tau_n=\Delta t>0$ and $r_n=r\in \N$ for all $n=1,\ldots,N_T$. Then, 
\begin{equation*}
	\vertiii{u-\ufd} = O(\Delta t^{\beta-3/2} + h^k)
\end{equation*}
as $\Delta t$ and $h$ decrease to 0.
\end{corollary}

\section{Numerical tests}
\label{sec:numerical_tests}

All the numerical tests are performed in Matlab, and make use of the VEM code available at \cite{Borio2020} for spatial discretization. For DG in time, we refer to~\cite{Antonietti-Mazzieri-Migliorini2020}. The meshes are generated using the code Polymesher~\cite{Talischi-Paulino-Pereira-Menezes2012}.

\subsection{Verification test}

 As verification test, we consider equation~\eqref{eq:pde} on $\Omega\times (0,T]=(0,1)^2\times (0,1]$, where $\nu =1$ and the loading term $f$ as well as the initial conditions $u_0,\, z_0$ are chosen so that 
\begin{gather}
	\label{eq:u_ex}
	u_{ex}(t,x_1,x_2)\coloneqq \sin(t^2)\sin(\pi x_1)\sin(\pi x_2)
\end{gather}
is the unique solution of the problem (see Figure~\ref{fig:exact_sol_ex1}). 

\begin{figure}
	\begin{subfigure}[b]{0.45\textwidth}
         \includegraphics[width=\textwidth]{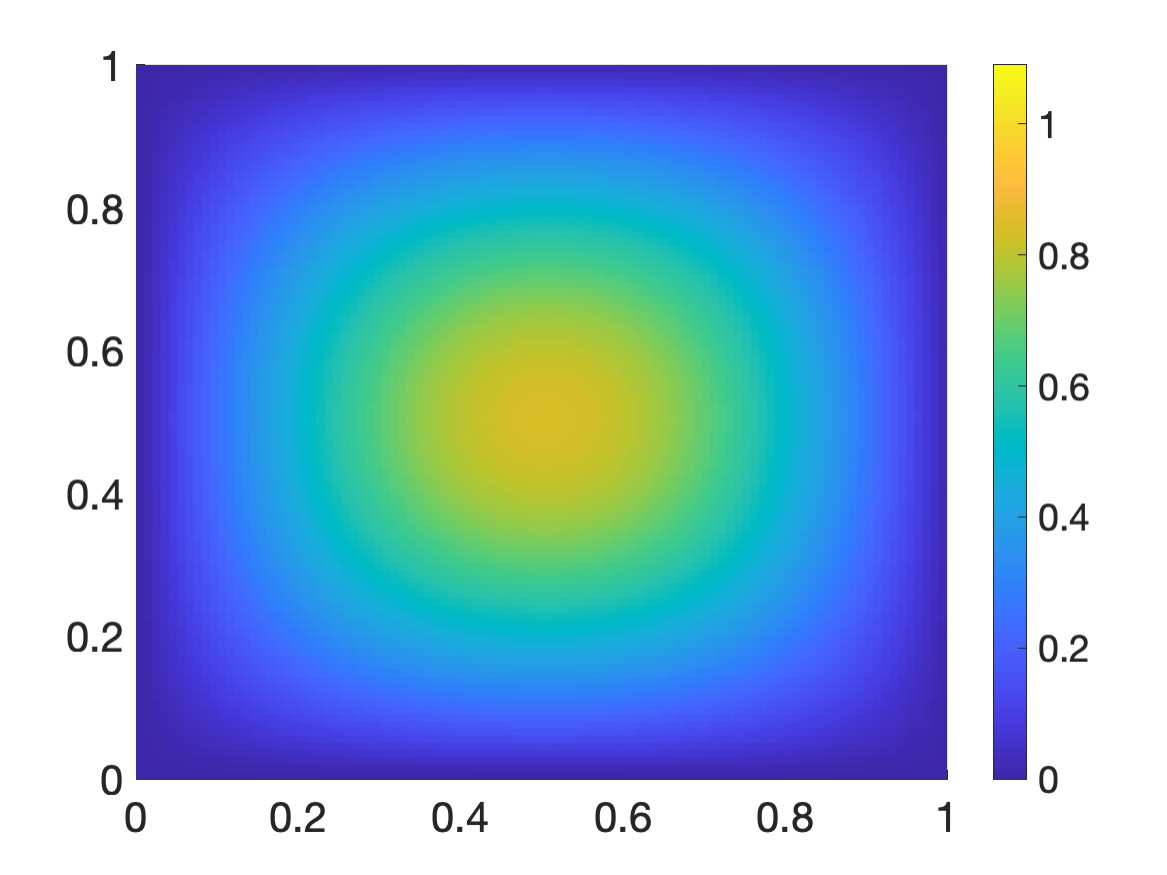}
         \caption{}
     \end{subfigure}
     \hfill
     \begin{subfigure}[b]{0.45\textwidth}
         \centering
         \includegraphics[width=\textwidth]{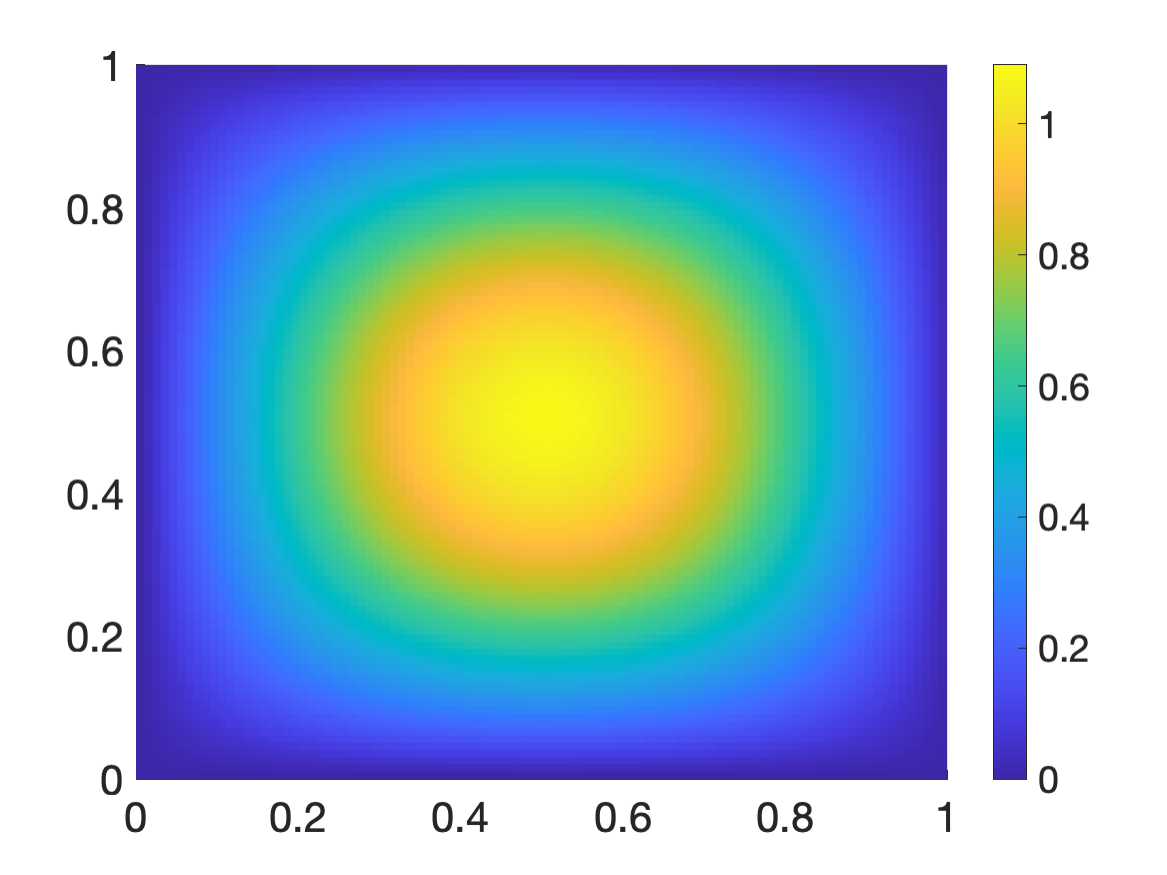}
         \caption{}
     \end{subfigure}
     \caption{(a) $u_{ex}$ at final time $T=1$; (b) $u_{ex,t}$ at final time $T=1$.}
        \label{fig:exact_sol_ex1}	
\end{figure}

First, we verify the convergence of the VEM-DG error as the time discretization refines.
We compute the VEM-DG solution $u_{h,\tau}$ applying the VEM of degree $k=4$ on the Veronoi mesh represented in Figure \ref{fig:fig2_ex1}(A), coupled with the DG method in time over uniform partitions of $[0,1]$ with decreasing length $\Delta t$ of the time-slabs and with varying polynomial degree $r=1,\, 2,\, 3$. Note that the case $r=1$ is not covered by the theory of Section~\ref{sec:DG_time}. In Figure \ref{fig:fig2_ex1}(B) we observe the expected decay of the error at final time, namely, $\vertiii{u_{ex}(T)-\ufd(T)}=O(\Delta t^{r-1/2})$ (see Corollary~\ref{cor:err_DG}).

\begin{figure}
    \begin{subfigure}[b]{0.45\textwidth}
        \includegraphics[width=\textwidth]{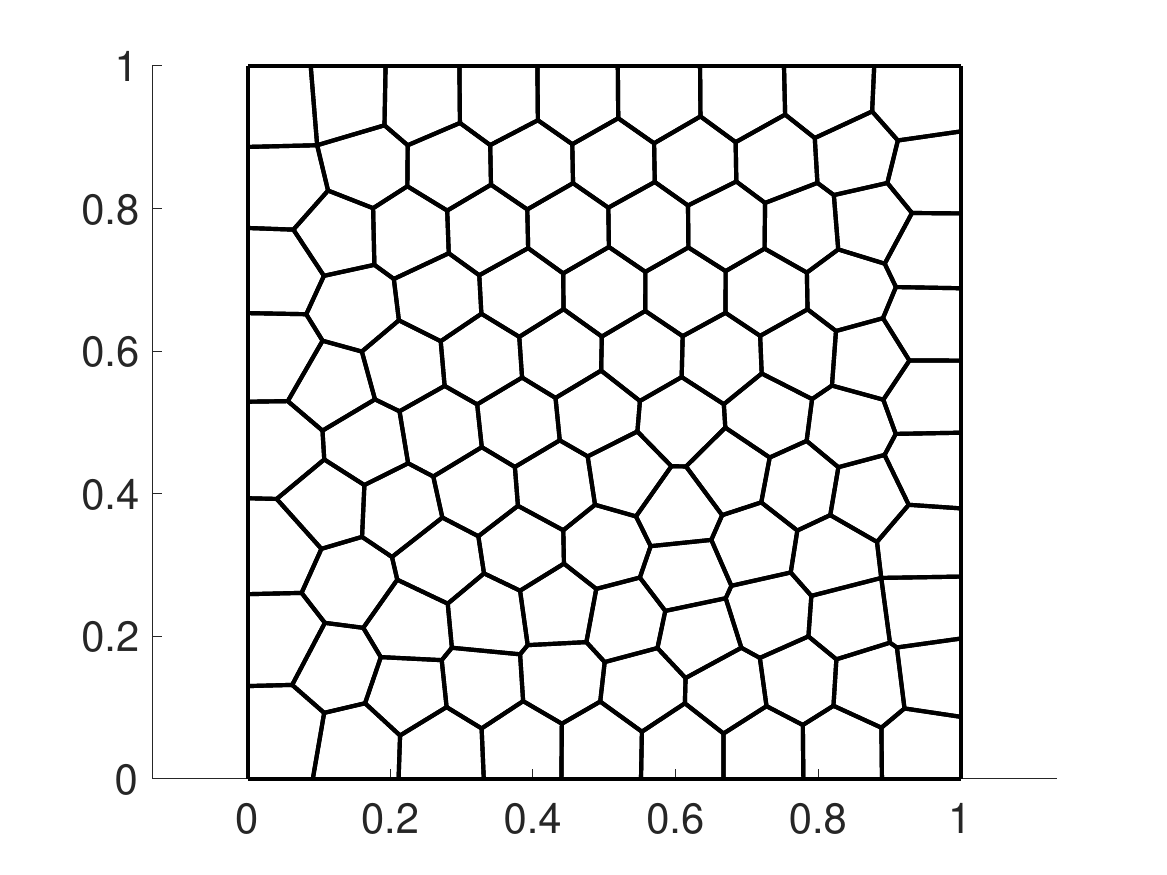}
        \caption{}         
    \end{subfigure}
    \hfill
    \begin{subfigure}[b]{0.45\textwidth}
        \includegraphics[width=\textwidth]{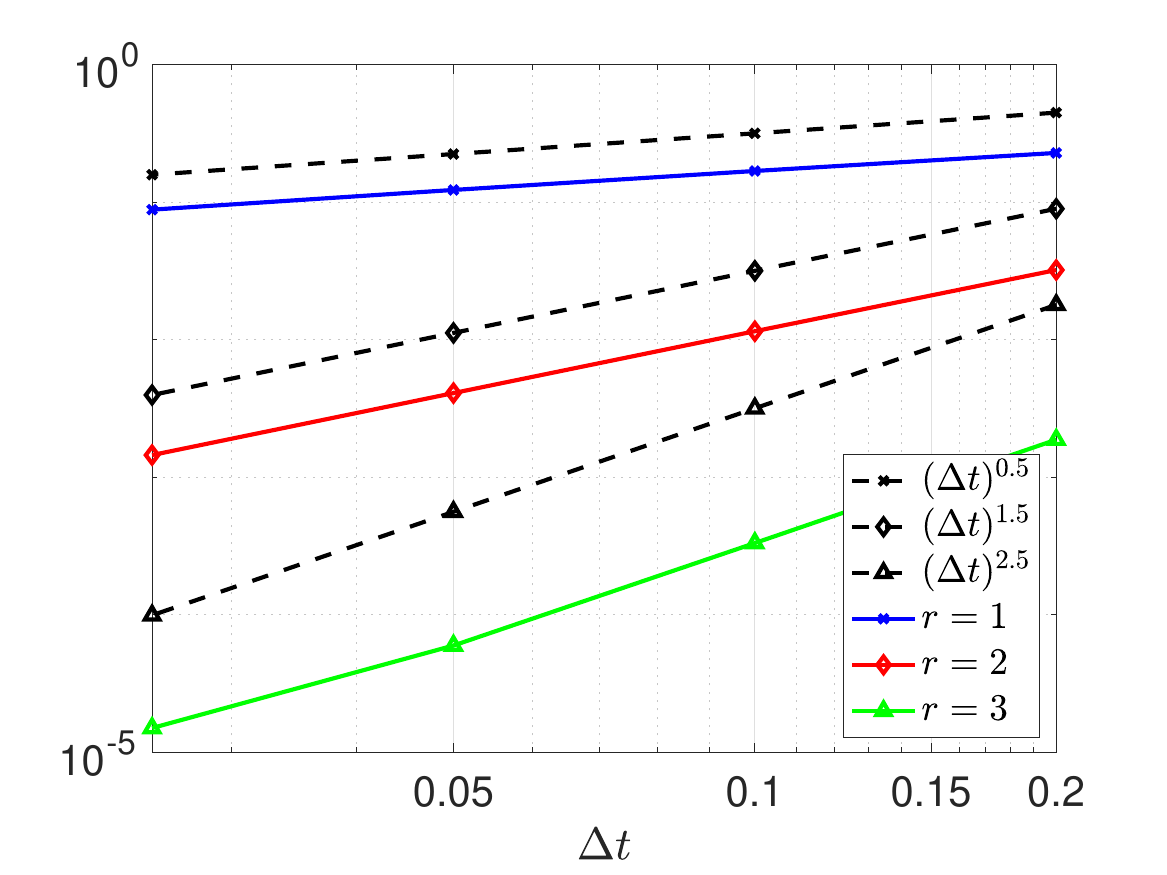}
        \caption{}         
     \end{subfigure}
    \caption{(a) Veronoi mesh with 100 polygonal elements; (b) $\vertiii{u_{ex}(T)-\ufd(T)}$, where $\ufd$ is computed using VEM of degree $k=4$ and DG of increasing degree $r=1,2,3$.}
    \label{fig:fig2_ex1}
\end{figure}

In the second experiment, we study the convergence of the VEM-DG error as the space discretization refines. To this end, we consider the DG approximation of degree $r=6$ over the uniform partition of $[0,1]$ with $\Delta t=0.1$, coupled with the VEM on different Veronoi meshes (see Figure \ref{fig:mesh}) and with increasing degree $k=1,\, 2,\, 3$. 
In Figure \ref{fig:error}(A) the expected behavior $\vertiii{u_{ex}(T)-\ufd(T)}=O(h^k)$ is observed (see Theorem~\ref{thm:VEM_estimate1}).

Finally, in the last experiment, we take $r=k$ and $h\sim\Delta t$. The error decay is depicted in Figure \ref{fig:error}(B), and it is in agreement with~\eqref{eq:error}.

\begin{figure}
	\begin{subfigure}[b]{0.45\textwidth}
         \includegraphics[width=\textwidth]{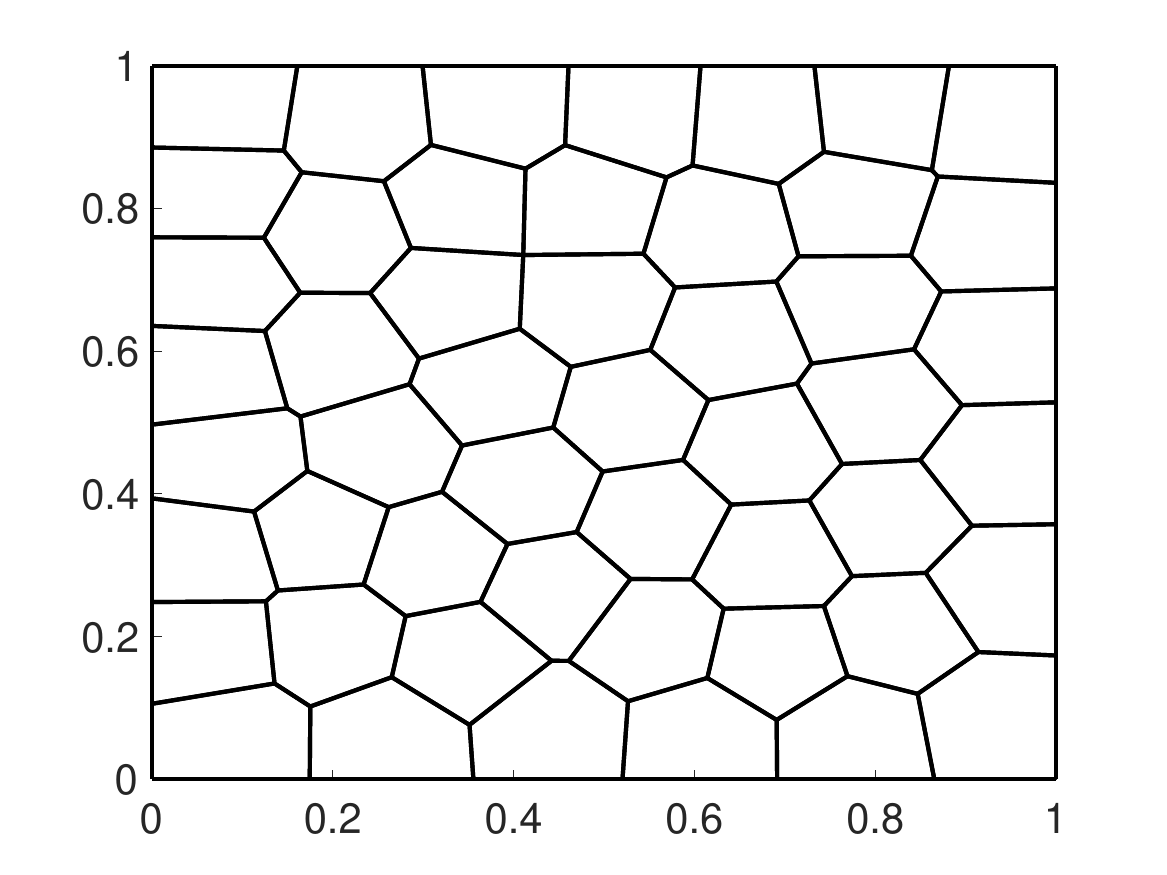}
         \caption{}
     \end{subfigure}
     \hfill
     \begin{subfigure}[b]{0.45\textwidth}
         \centering
         \includegraphics[width=\textwidth]{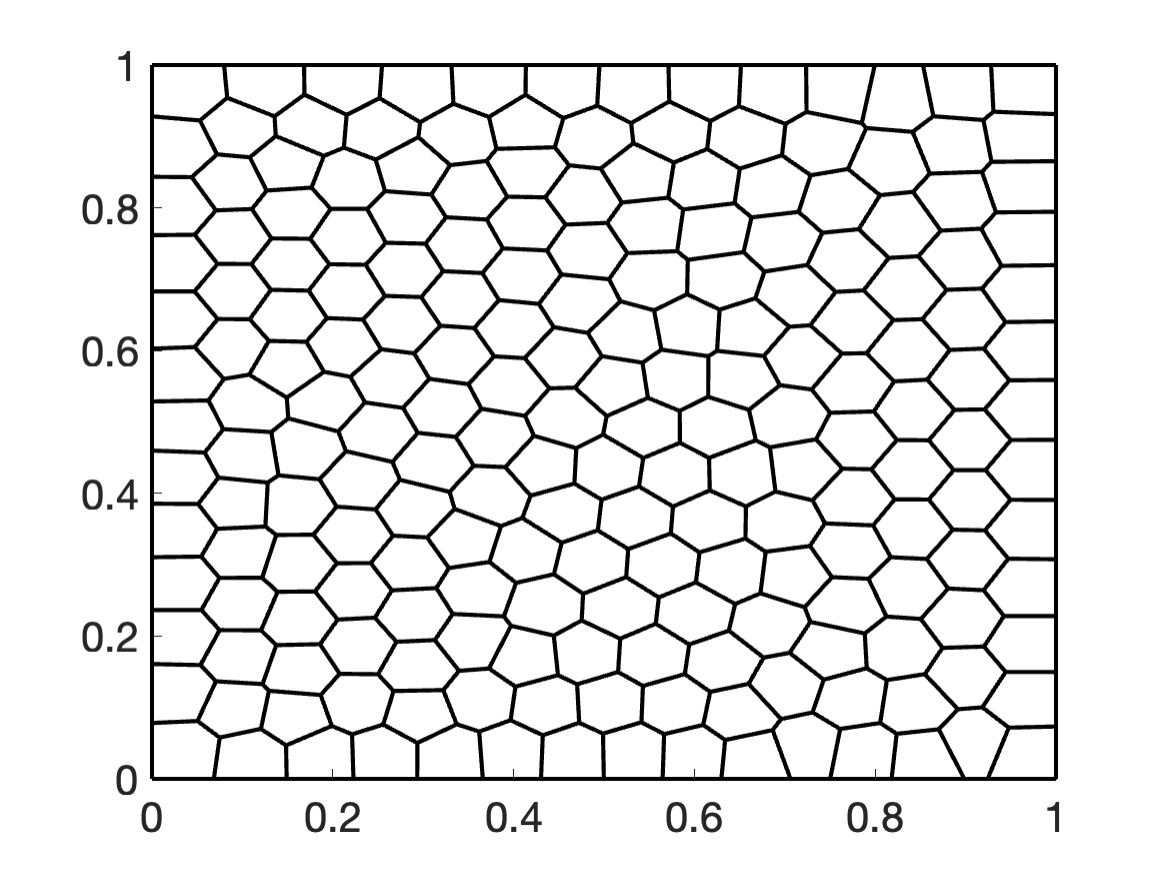}
         \caption{}
     \end{subfigure}
     \begin{subfigure}[b]{0.45\textwidth}
         \includegraphics[width=\textwidth]{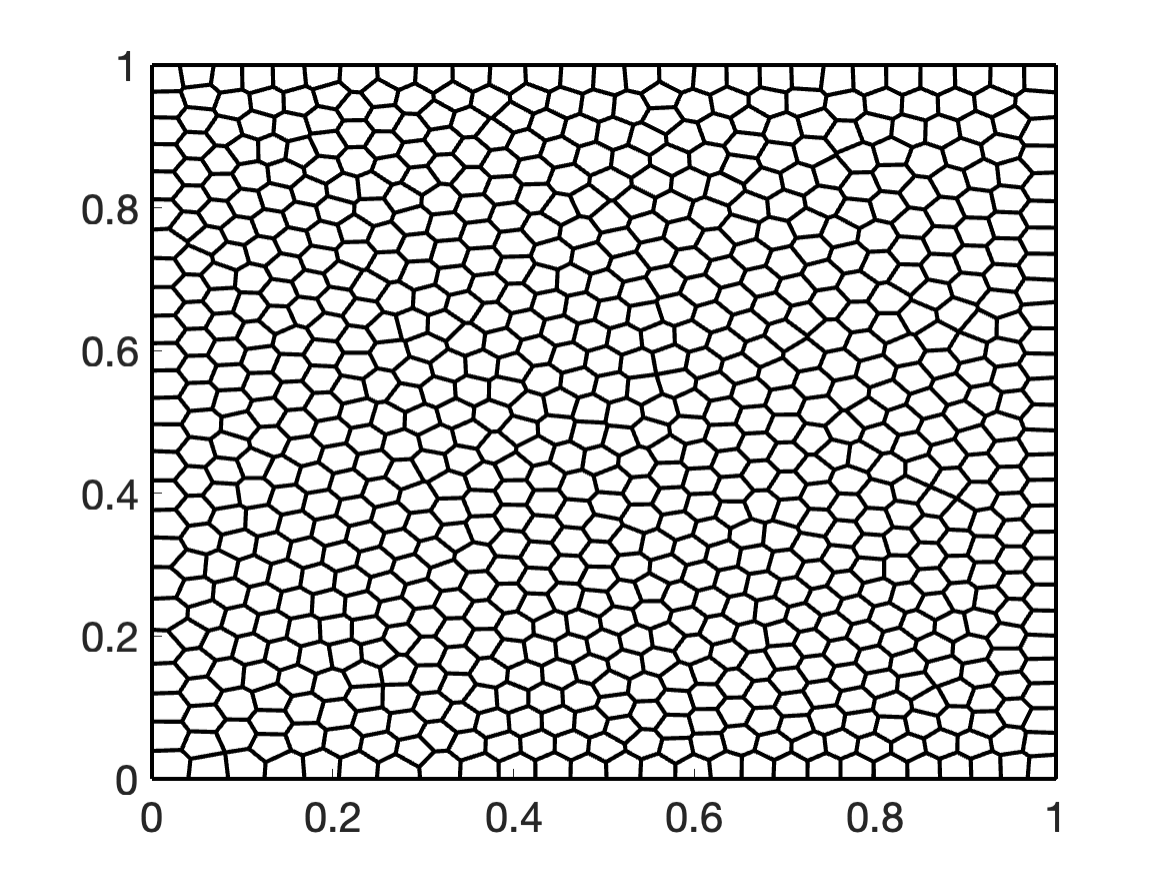}
         \caption{}
     \end{subfigure}
     \hfill
     \begin{subfigure}[b]{0.45\textwidth}
         \centering
         \includegraphics[width=\textwidth]{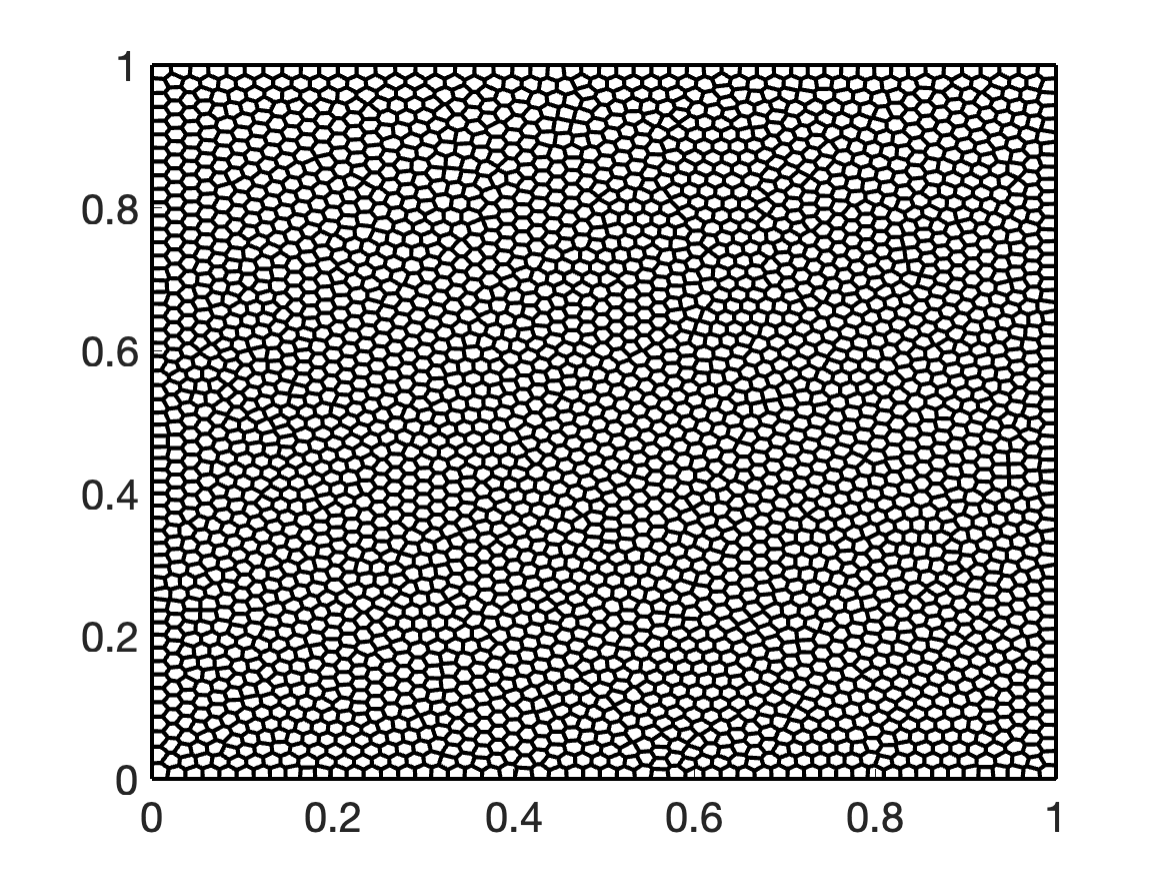}
         \caption{}
     \end{subfigure}
     \caption{Veronoi meshes with 50 (top left) 200 (top right) 800 (bottom left) and 3200 (bottom right) elements.}
        \label{fig:mesh}	
\end{figure}

\begin{figure}
	\begin{subfigure}[b]{0.45\textwidth}
        		\includegraphics[width=\textwidth]{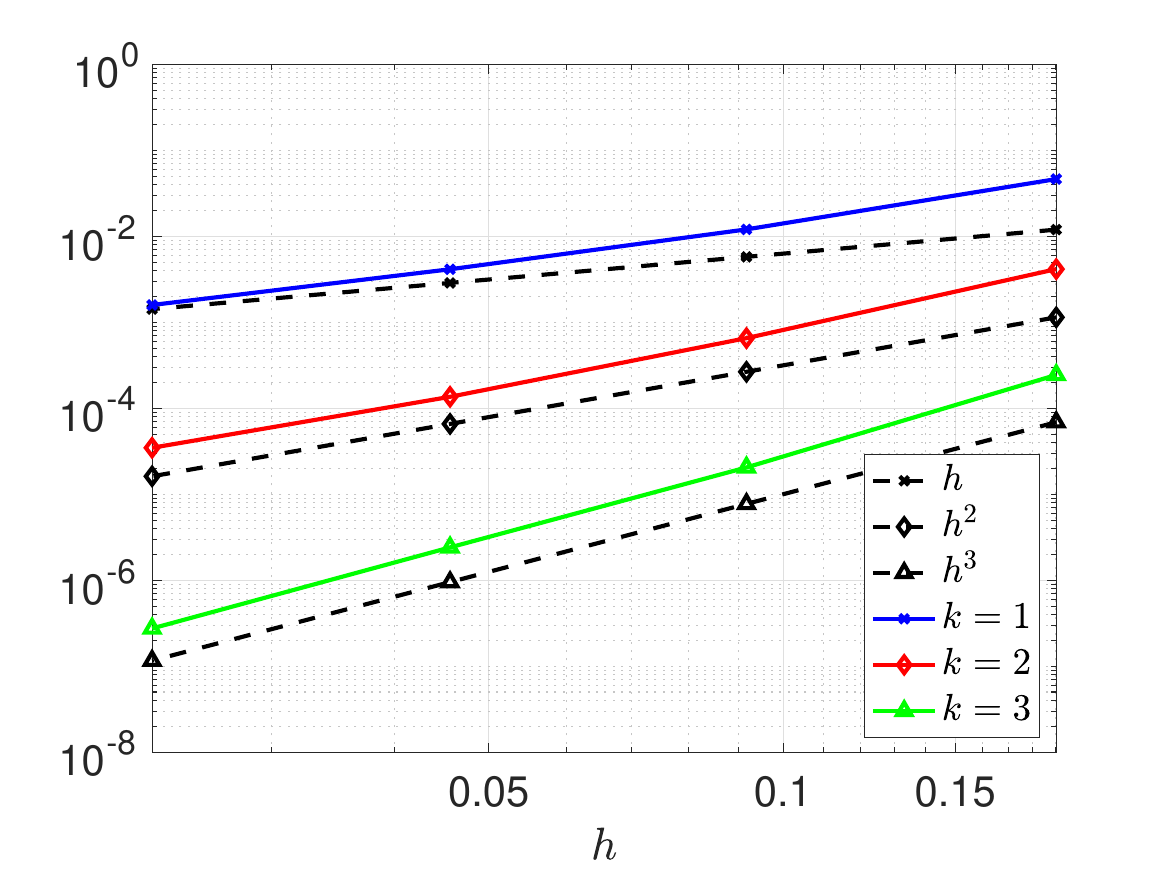}
         	\caption{}         
     	\end{subfigure}
     \hfill
     \begin{subfigure}[b]{0.45\textwidth}
        		\includegraphics[width=\textwidth]{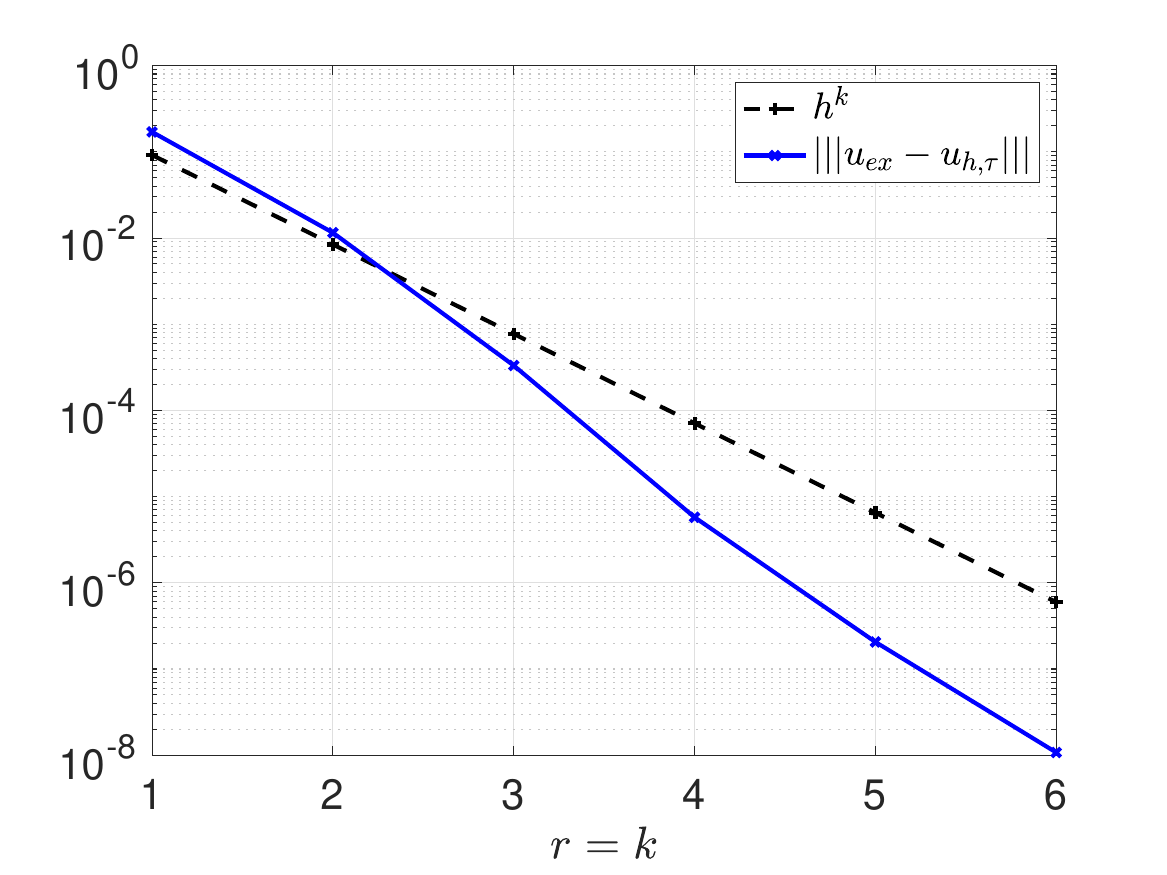}
         	\caption{}         
     	\end{subfigure}
	\caption{(a) $\vertiii{u_{ex}(T)-\ufd(T)}$, where $\ufd$ is computed using DG of degree $r=6$ on a uniform partition of the time interval with $\Delta t=0.01$ and VEM of increasing degree $k=1,\, 2,\, 3$.  (b) $\vertiii{u_{ex}(T)-\ufd(T)}$, where $\ufd$ is computed using DG on a uniform partition of the time interval with $\Delta t=0.1$ and VEM on the Veronoi mesh with 200 elements (see Figure~\ref{fig:mesh} (B), with equal degree in time and space.}
         \label{fig:error}
\end{figure}

\subsection{Validation test}

The second experiment deals with a more realistic scenario, and aims at investigating the performances of the proposed numerical scheme in the non-dissipative case, which is not covered by the theory here developed. In particular, we consider problem~\eqref{eq:pde} with $\nu=0$, initial data $u_0\equiv 0$, $z_0\equiv 0$ and loading term
\begin{equation}
	\label{eq:f}
	f(t,x)=\left\{\begin{array}{ll}
	0&\textrm{ for }t<0.1\\
	100\, e^{ -\frac{(x-x_0)^4}{s^2}} &\textrm{ else}
	\end{array}\right.
\end{equation}
representing a smooth impulse centered at $x_0=(0.05,0.05)$, with $s=0.025$ (see Figure~\ref{fig:f_uex} (A)). For such example, there is no analytical solution. Hence, we refer to an overkilled solution computed by means of the VEM of degree 2 on a spatial mesh with 3200 elements coupled with DG for time discretization, with polynomial degree 2 and $\Delta t=1/320$ (see Figure~\ref{fig:f_uex} (B)).

In Figure~\ref{fig:sol_test2} we represent the snapshots at final time $T=1$ of the approximated solution obtained by means of the proposed VEM-DG strategy (the parameters for time integration are $\Delta t=1/20$ and $r=2$), compared with the approximations produced using the Newmark method for increasing $\Delta t$. Note that the numerical scheme for the space integration is the same as in the reference solution. We can observe that the discrete solution computed with Newmark is affected by spurious oscillations.
In Figure~\ref{fig:sol_sensor} we report the computed time history of the displacement on a receiver located at $(0.5,0.5)$. It is clear that the DG-VEM approximation is more accurate than those computed with the Newmark method.

\begin{figure}
	\begin{subfigure}[b]{0.45\textwidth}
        		\includegraphics[width=\textwidth]{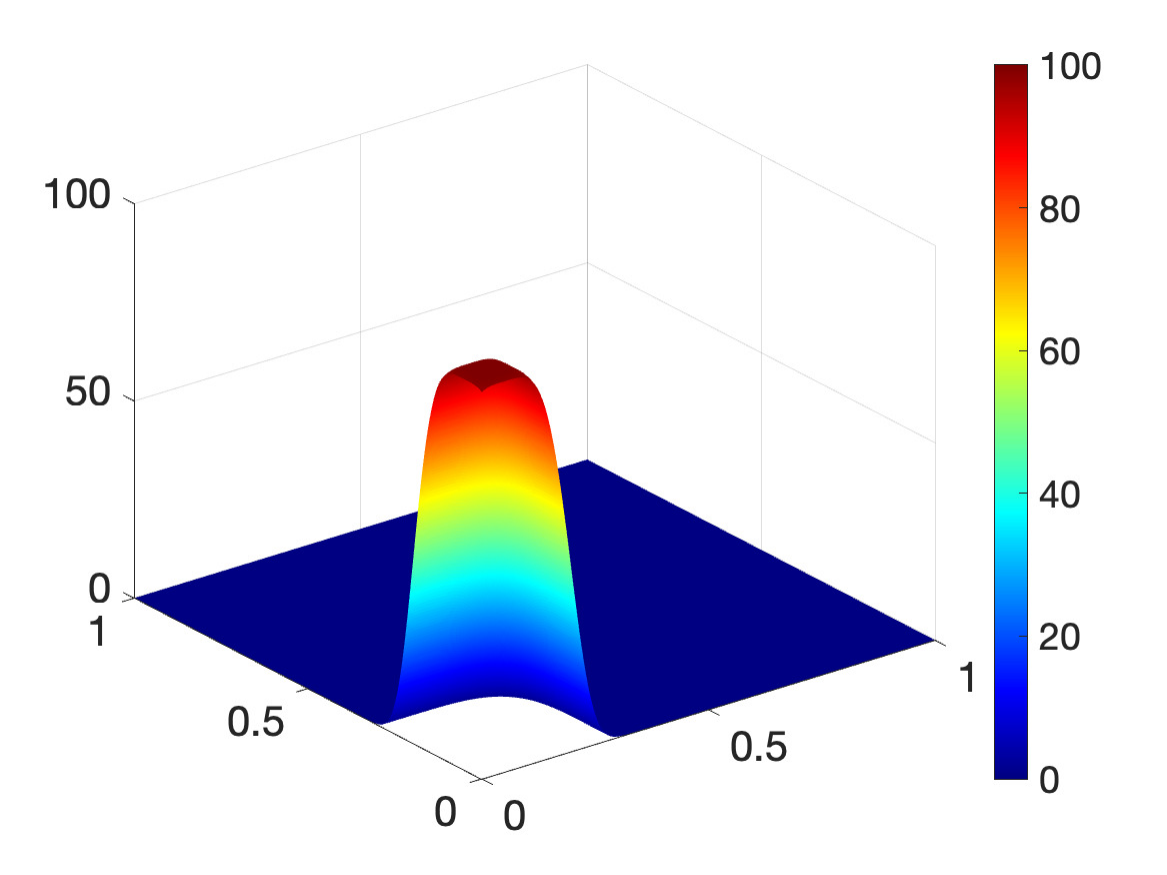}
         	\caption{}         
     	\end{subfigure}
     \hfill
     \begin{subfigure}[b]{0.45\textwidth}
        		\includegraphics[width=\textwidth]{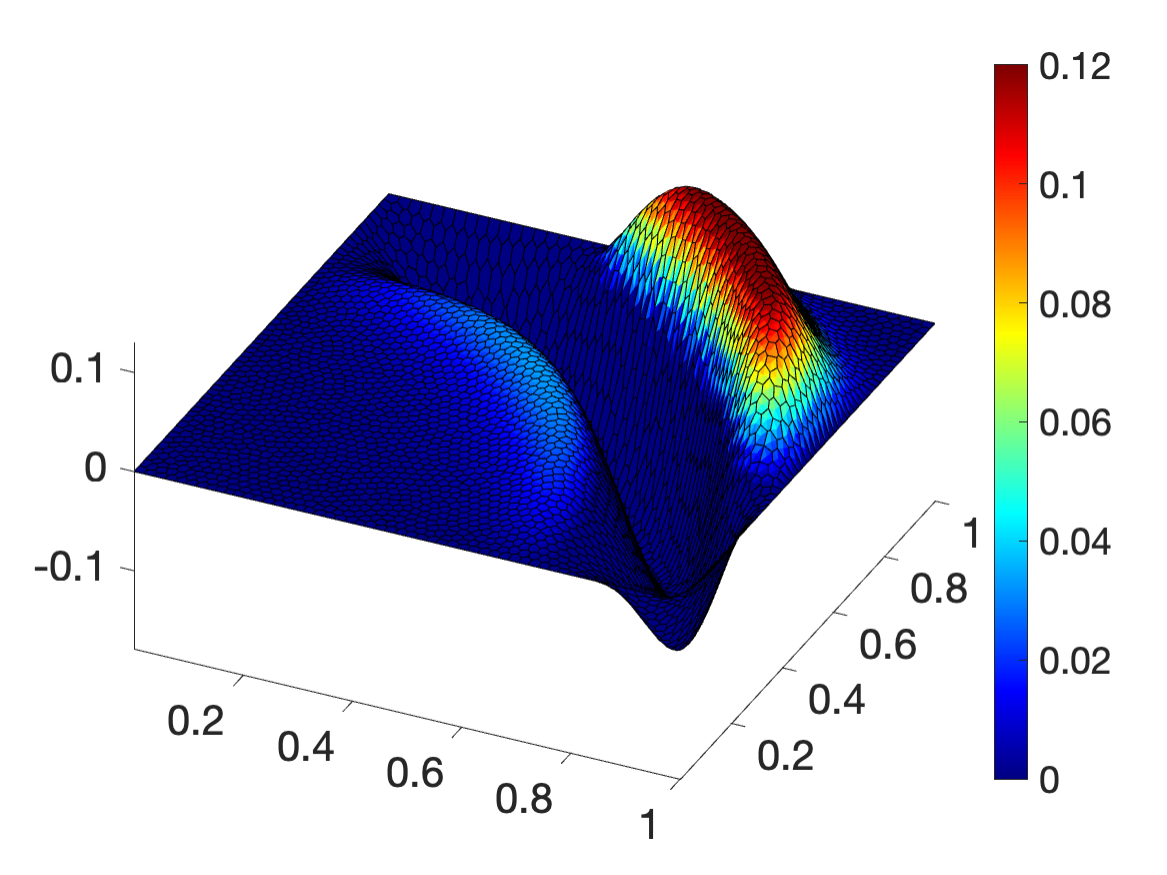}
         	\caption{}         
     	\end{subfigure}
	\caption{(a) Loading term~\eqref{eq:f}. (b) Reference solution.}
         \label{fig:f_uex}
\end{figure}

\begin{figure}
	\begin{subfigure}[b]{0.45\textwidth}
         \includegraphics[width=\textwidth]{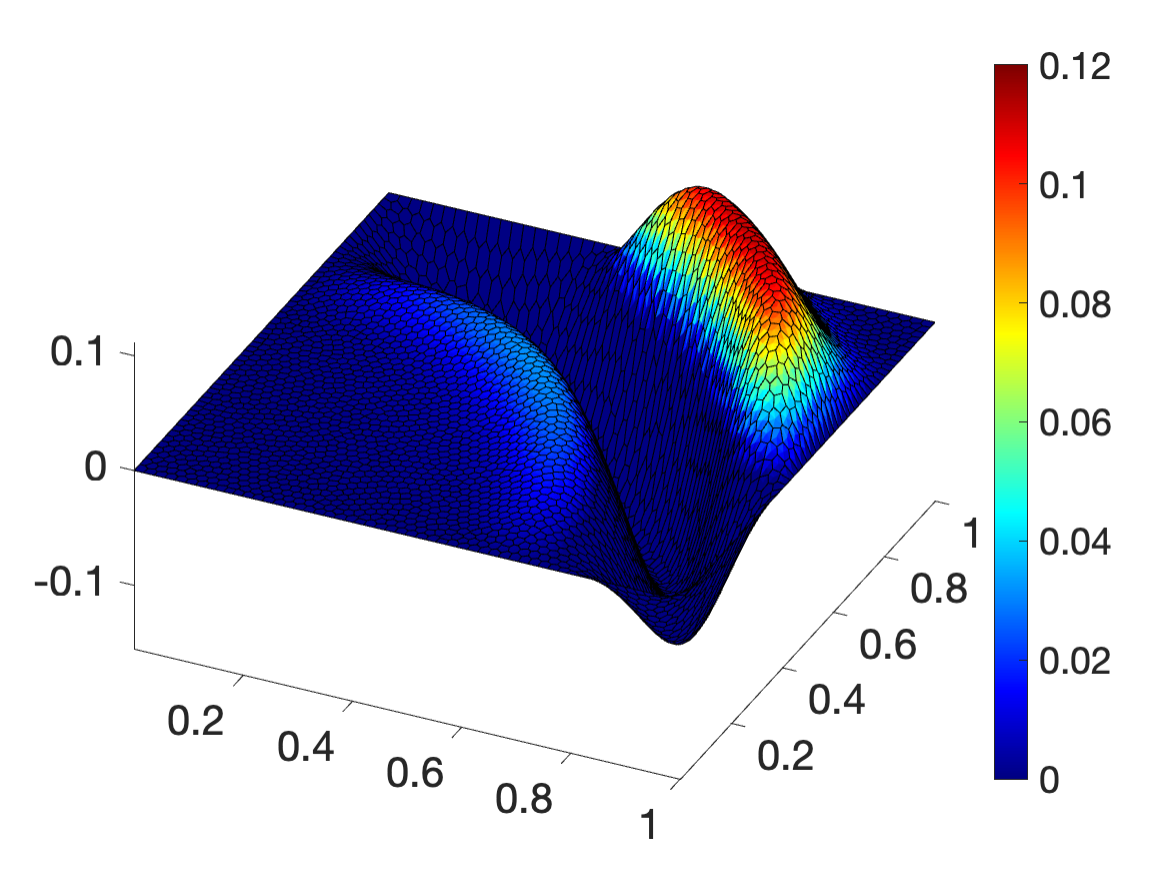}
         \caption{}
     \end{subfigure}
     \hfill
     \begin{subfigure}[b]{0.45\textwidth}
         \centering
         \includegraphics[width=\textwidth]{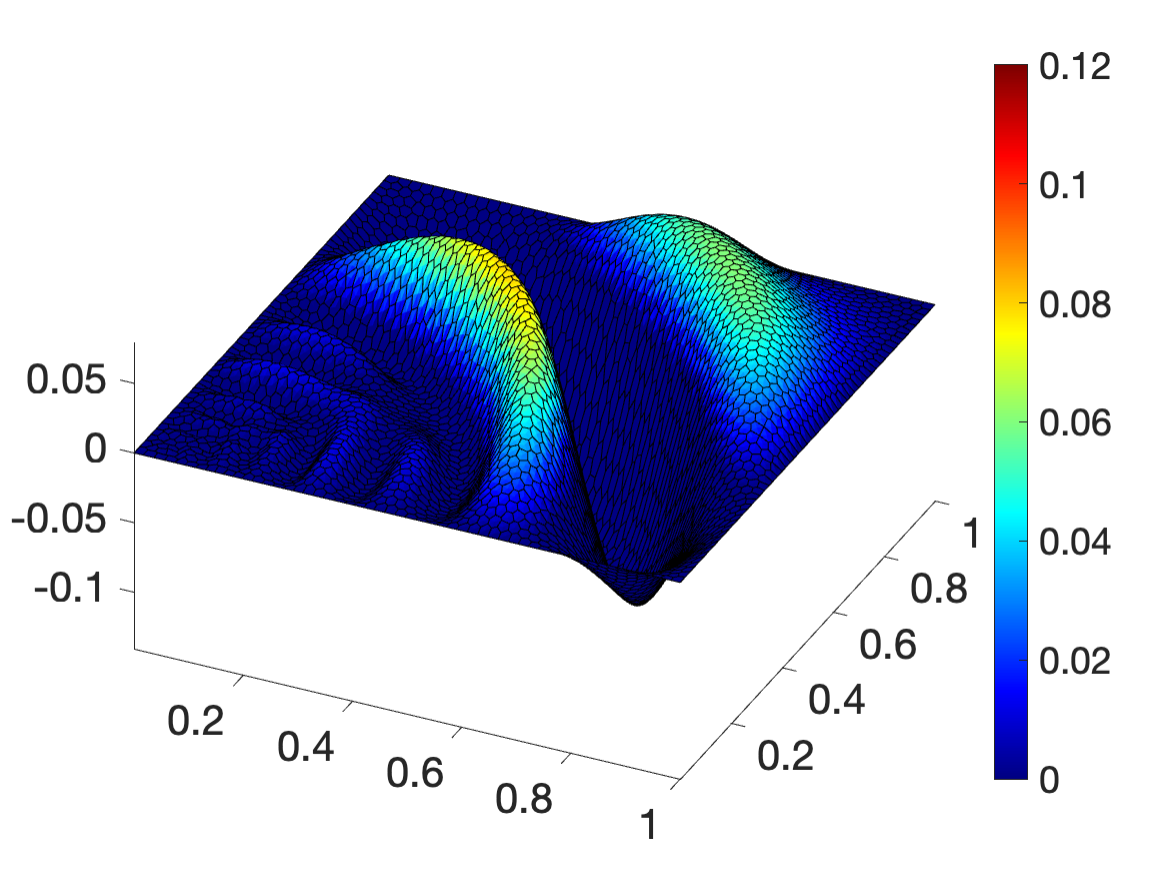}
         \caption{}
     \end{subfigure}
     \begin{subfigure}[b]{0.45\textwidth}
         \includegraphics[width=\textwidth]{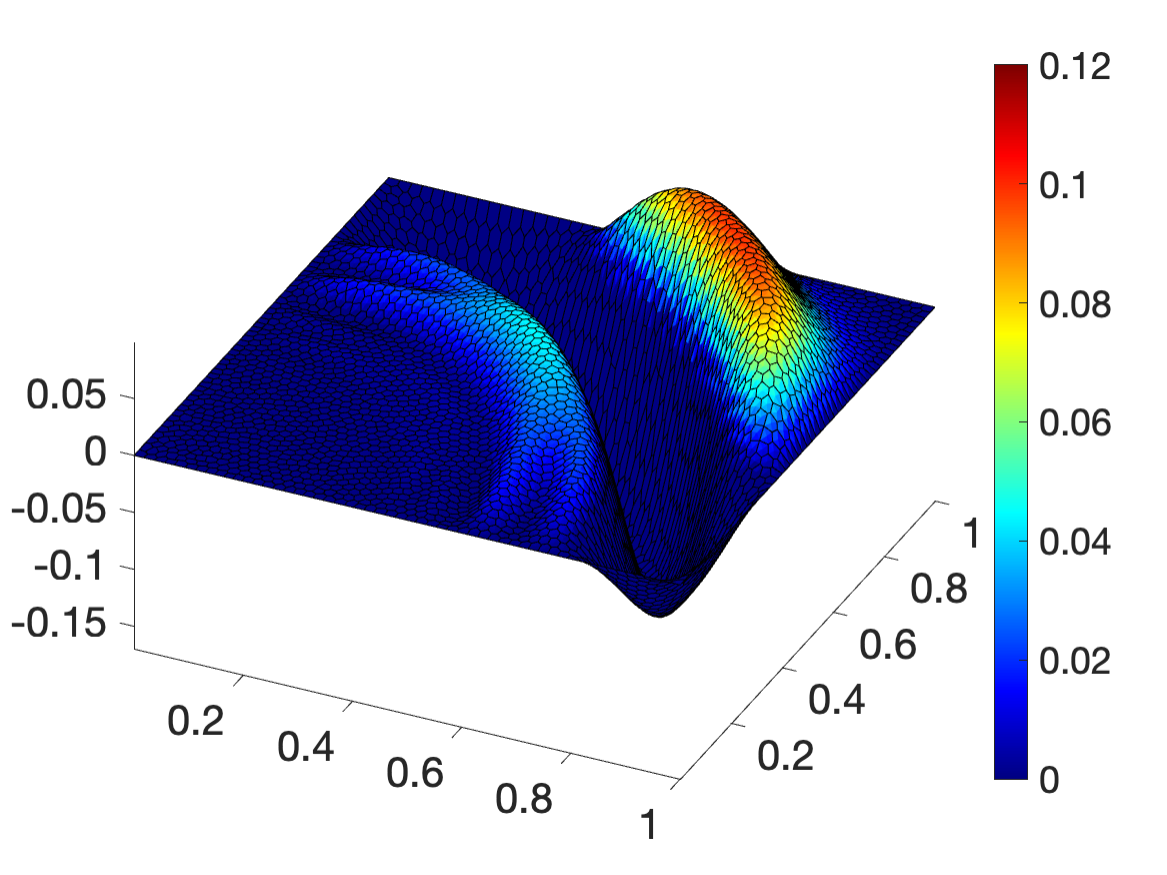}
         \caption{}
     \end{subfigure}
     \hfill
     \begin{subfigure}[b]{0.45\textwidth}
         \centering
         \includegraphics[width=\textwidth]{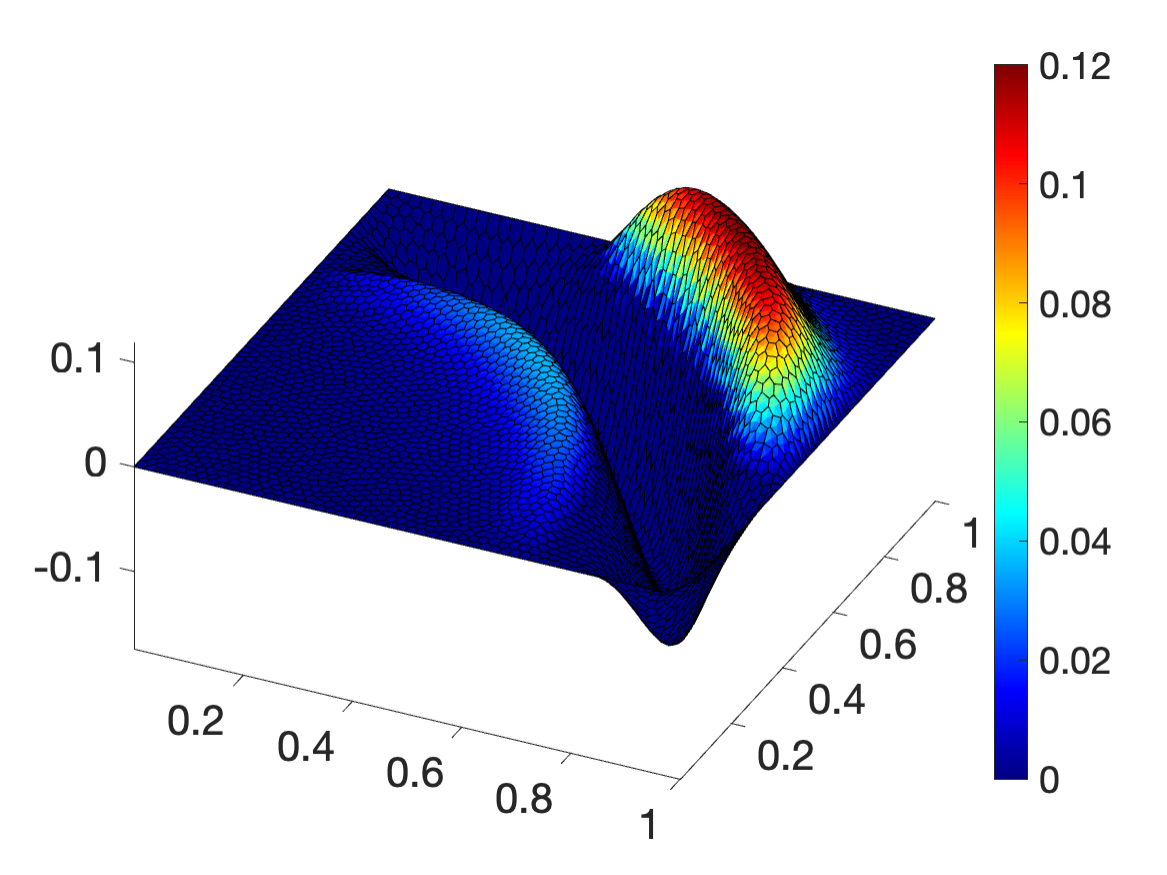}
         \caption{}
     \end{subfigure}
     \caption{Fully discrete solution computed by means of the proposed DG-VEM strategy with $\Delta t=1/20$ and $r=2$ (a) compared with the numerical approximation obtained by means of VEM in space coupled with Newmark for time integration, with $\Delta t=1/20$ (b), $\Delta t=1/40$ (c) and $\Delta t=1/80$ (d).}
        \label{fig:sol_test2}	
\end{figure}

\begin{figure}
    \begin{center}
	\includegraphics[width=0.45\textwidth]{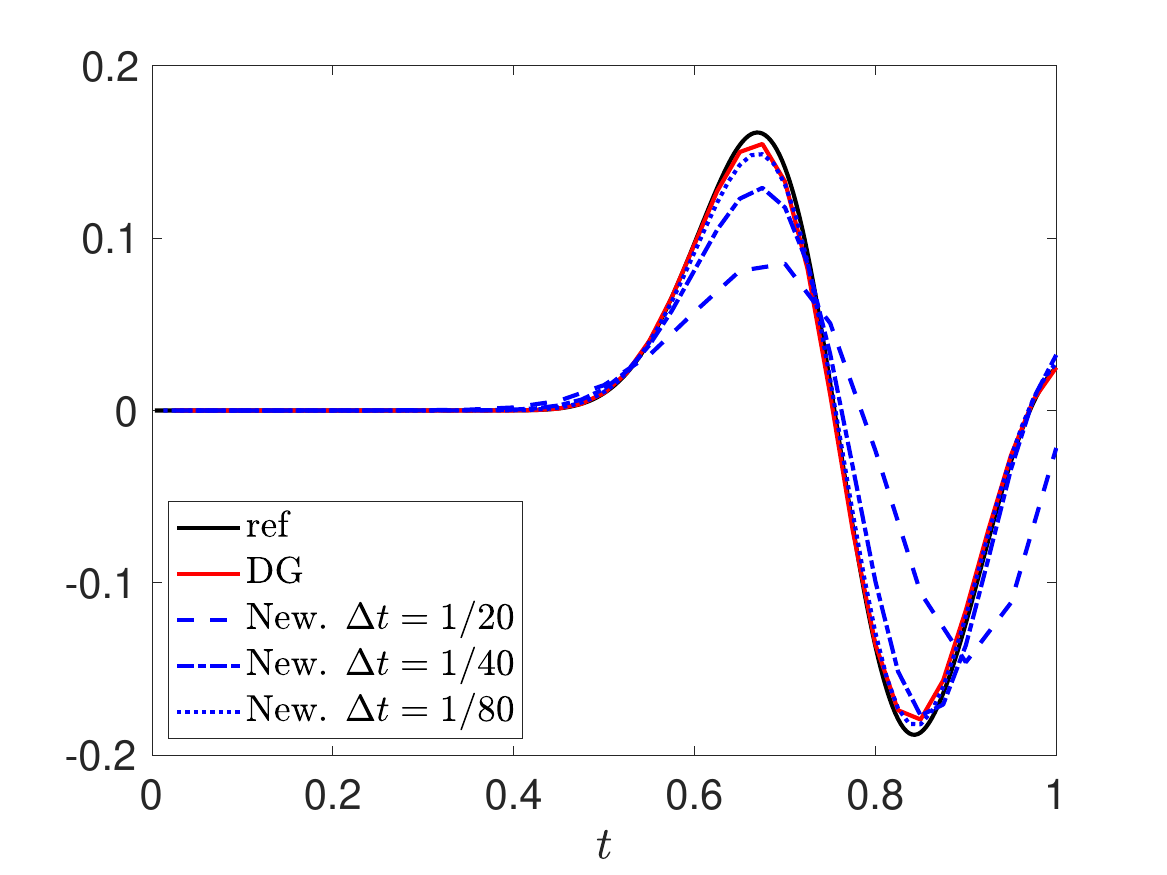}
        	\caption{Computed time history of the displacement on a receiver located at $(0.5,0.5)$. The black line represents the reference solution. The red line represents the DG-VEM solution. The dashed blue lines represent the solutions computed with the Newmark method for increasing $\Delta$.}         
        	\label{fig:sol_sensor}
    \end{center}
\end{figure}

\appendix

\section{Representation formula for the semi-discrete solution}
\label{sec:repr_form}

\begin{theorem}
\label{thm:VEM_wp}
The unique solution to problem~\eqref{eq:pde_weak_sd} is given by
\begin{gather}
	\label{eq:uh_expression}
	\begin{aligned}
	u_h(t)&\coloneqq \sum_{n=1}^{N_h} \gamma_n(t) w_h^{(n)},
	\end{aligned}
\end{gather}
where $\{w_h^{(n)}\}_{n=1}^{N_h}$ is the basis of $W_h$ orthonormal with respect to $m_h(\bullet,\bullet)$ fulfilling, for all $v_h\in W_h$ and for all $n=1,\ldots,N_h$
\begin{gather*}
	a_h(w_h^{(n)},v_h) = \lambda_h^{(n)} m_h(w_h^{(n)},v_h),
\end{gather*}
with $0<\lambda_h^{(1)}\leq \cdots\leq \lambda_h^{(N_h)}$, and the $n$-th coefficient in the eigen-expansion of $u_h(t)$~\eqref{eq:uh_expression} is given by 
\begin{gather}
	\label{eq:uh_coeff}
	\begin{aligned}
	\gamma_n(t)&\coloneqq 
	e^{-\nu/2 t}\Bigg[m_h(u_{h,0}, w_h^{(n)})\cos(\omega_h^{(n)} t)
	+\frac{1}{\omega_h^{(n)}} m_h(z_{h,0}, w_h^{(n)})\sin(\omega_h^{(n)} t)\\
	&\quad + \frac{1}{\omega_h^{(n)}} \int_0^t e^{-\nu/2 (t-s)} \sin(\omega_h^{(n)} (t-s)) \sp{f_h(s)}{w_h}\, ds\Bigg]
	\end{aligned}
\end{gather}
where $\omega_h^{(n)}\coloneqq\sqrt{\lambda_h^{(n)}-\frac{\nu^2}{4}}$, with $\nu$ being small enough so that $\lambda_h^{(n)}-\frac{\nu^2}{4}>0$ for all $n=1,\ldots, N_h$.
Moreover, for all $t\in (0,T)$, there holds
\end{theorem}

To prove Theorem~\ref{thm:VEM_wp} we need two auxiliary results.
\begin{lemma}
\label{lem:sol_hom}
Let 
\begin{gather}
	\label{eq:Fn}
	F_n(t)\coloneqq e^{-\nu/2 t}\Big(c_n \cos(\omega_h^{(n)} t)+\frac{d_n}{\omega_h^{(n)}} \sin(\omega_h^{(n)} t)\Big), 
\end{gather}	
with $c_n,\, d_n>0$. Then, there holds
\begin{gather}
	\label{eq:sol_hom}
	\ddot F_n(t) + \nu \dot F_n(t) + \lambda_h^{(n)} F_n(t) =0,
\end{gather}
where $\lambda_h^{(n)}, \omega_h^{(n)}$ have been defined in Theorem~\ref{thm:VEM_stability}.
\end{lemma}

\begin{proof}
Equation~\eqref{eq:sol_hom} follows by observing that
\begin{gather*}
	\begin{aligned}
		\dot F(t)&=
		e^{-\nu/2 t}\Bigg[ \Big(-\frac{\nu}{2} c_n+d_n\Big) \cos(\omega_h^{(n)} t)
		+  \Big(-\frac{\nu}{2} \frac{c_n}{\omega_h^{(n)}}-c_n\omega_h^{(n)}\Big) \sin(\omega_h^{(n)} t)\Bigg],\\
		\ddot F(t) &= e^{-\nu/2 t}\Bigg[
		\Big(\frac{\nu^2}{4}c_n-\nu d_n-c_n( \omega_h^{(n)})^2\Big) \cos(\omega_h^{(n)} t)\\
		&\quad
		+\Big(\frac{\nu^2}{4}\frac{d_n}{\omega_h^{(n)}}+\nu c_n\omega_h^{(n)}-\omega_h^{(n)}d_n \Big) \sin(\omega_h^{(n)} t)
		\Bigg]. 
	\end{aligned}
\end{gather*}
\end{proof}

\begin{lemma}
\label{lem:sol_part}
Let 
\begin{gather}
	\label{eq:Gn}
	G_n(t)\coloneqq\int_{0}^t g_n(t,s)\, ds, 
\end{gather}
with 
\begin{gather}
	\label{eq:gn}
g_n(t,s)\coloneqq \frac{1}{\omega_h^{(n)}} e^{-\nu/2 (t-s)} \sin(\omega_h^{(n)} (t-s)) \sp{f_h(s)}{w_h}.
\end{gather}
Then, there holds
\begin{gather}
	\label{eq:sol_part}
	\ddot G_n(t) + \nu \dot G_n(t) + \lambda_h^{(n)} G_n(t) =\sp{f_h(t)}{w_h},
\end{gather}
where $\lambda_h^{(n)}, \omega_h^{(n)}$ have been defined in Theorem~\ref{thm:VEM_stability}.
\end{lemma}

\begin{proof}
We note that
\begin{gather*}
	\begin{aligned}
		\frac{d}{dt} \int_0^t g(s,t)\, ds &= g(t,t) - g(0,t) 
			+ \int_0^t\frac{\partial}{\partial t} g(s,t)\, ds,\\
		\frac{d^2}{dt^2} \int_0^t g(s,t)\, ds &=
			\frac{d}{dt} \Big(g(t,t)-g(0,t)\Big) 
			+ \frac{\partial}{\partial t} g(s,t)|_{s=t}
			-\frac{\partial}{\partial t} g(s,0)
			+ \int_0^t\frac{\partial^2}{\partial t^2} g(s,t)\, ds.
	\end{aligned}
\end{gather*}
Equation~\eqref{eq:sol_part} follows by choosing $f(0)=0$ and observing that
\begin{gather*}
	\begin{aligned}
	\dot G(t) = &\int_0^t e^{-\nu/2 (t-s)}\Bigg[
	-\frac{\nu}{2\omega_h^{(n)}}\sin(\omega_h^{(n)} (t-s)) + \cos(\omega_h^{(n)} (t-s))
	\Bigg] \sp{f_h(s)}{w_h}\, ds,
	\end{aligned}
\end{gather*}
and
\begin{gather*}
	\begin{aligned}
	\ddot G(t) &= \sp{f_h(t)}{w_h}
	+ \int_0^t e^{-\nu/2 (t-s)} \sp{f_h(s)}{w_h} \\
        &\Bigg[
	-\nu\cos(\omega_h^{(n)} (t-s))+\Bigg(\frac{\nu^2}{4\omega_h^{(n)}}-\omega_h^{(n)}\Bigg)
	\sin(\omega_h^{(n)} (t-s)) 
	\Bigg]\, ds.
	\end{aligned}
\end{gather*}
\end{proof}

\begin{proof}[Proof of Theorem A.1]
Since $\{w_h^{(n)}\}_{n=1}^{N_h}$ is the basis of $W_h$, it is enough to verify that~\eqref{eq:uh_expression} fulfills problem~\eqref{eq:pde_weak_sd} for all test functions $v_h=w_h^{(n)}$, with $n=1,\ldots, N_h$.  
Observe that
\begin{gather*}
	\begin{aligned}
	a_h(u_h(t),w_h^{(n)})&=\sum_{m=1}^{N_h}\gamma_n(t) a_h(w_h^{(m)},w_h^{(n)})\\
	&=\left\{\begin{array}{ll}
	0,&\textrm{ if }n\neq m,\\
	\gamma_n(t)\lambda_h^{(n)} \normL{w_h^{(n)}}^2 = \lambda_h^{(n)}\gamma_n(t), &\textrm{ if }n= m,
	\end{array}\right.
	\end{aligned}
\end{gather*}
and, analogously, $m_h(u_h(t),w_h^{(n)}) = \gamma_n(t)$.
Then,
\begin{align}
	\nonumber
	&m_h(u_{h,tt}(t),w_h^{(n)})+\nu m_h(u_{h,t}(t),w_h^{(n)})+a_h(u_{h}(t),w_h^{(n)})\\
	\nonumber
	&=\frac{d^2}{dt^2} m_h(u_{h}(t),w_h^{(n)}) + \nu \frac{d}{dt} m_h(u_{h}(t),w_h^{(n)}) + a_h(u_{h}(t),w_h^{(n)})\\
	\label{eq:proofa}
	& = \frac{d^2}{dt^2}  \gamma_n(t) + \nu \frac{d}{dt} \gamma_n(t) +  \lambda_h^{(n)}\gamma_n(t).
\end{align}
We conclude that $\eqref{eq:proofa}= \sp{f_h(t)}{w_h}$ by applying Lemma~\ref{lem:sol_hom}, since $\gamma_n(t)=F_n(t)+G_n(t)$, where $F_n(t)$ is of the form~\eqref{eq:Fn} - the constants being fixed so that $u_h(t)$ fulfills the initial conditions $u_h(0)=u_{h,0}$, $u_{h,t}(0)=z_{h,0}$, namely,
\begin{gather*}
	c_n = m_h(u_{h,0}, w_h^{(n)}),\quad
	d_n = \frac{1}{\omega_h^{(n)}} m_h(z_{h,0}, w_h^{(n)}),
\end{gather*}
and by applying Lemma~\ref{lem:sol_part}, since $G_n(t)$ is of the form~\eqref{eq:Gn}-\eqref{eq:gn}.
\end{proof}

\bibliographystyle{plain}

\end{document}